\documentclass[a4paper,12pt]{amsart}

\usepackage{amssymb,xspace,amscd,yhmath,mathdots,txfonts,
mathrsfs, youngtab,stmaryrd}

\usepackage[matrix,arrow,curve]{xy}\CompileMatrices

\usepackage{hyperref}

\usepackage{yfonts}[1998/10/03]

\DeclareMathAlphabet{\mathpzc}{OT1}{pzc}{m}{it}

\numberwithin{equation}{section}

\theoremstyle{plain}

\newtheorem{thm}{Theorem}[section]

\newtheorem{lem}[thm]{Lemma}

\newtheorem{cor}[thm]{Corollary}

\newtheorem{prop}[thm]{Proposition}

\theoremstyle{definition}

\newtheorem{defn}{Definition}[section]
\newtheorem{exam}[thm]{Example}
\newtheorem{ntz}{Notation}[section]

\newtheorem{rmk}[thm]{Remark}

\DeclareMathAlphabet{\mathpzc}{OT1}{pzc}{m}{it}

\newcommand\ot{\mathfrak{o}}

\DeclareMathOperator{\R}{{\mathbb{R}}}

\DeclareMathOperator{\kt}{{\kappaup}}

\DeclareMathOperator{\Li}{\mathfrak{L}}

\DeclareMathOperator{\rad}{\mathrm{rad}}

\DeclareMathOperator{\Gf}{\mathbf{G}}

\DeclareMathOperator{\ad}{\mathrm{ad}}

\DeclareMathOperator{\ses}{\mathfrak{s}}
\DeclareMathOperator{\ct}{\mathfrak{c}}

\DeclareMathOperator{\zt}{\mathfrak{z}}

\DeclareMathOperator{\Z}{\mathbb{Z}}

\DeclareMathOperator{\ft}{\mathfrak{f}}

\DeclareMathOperator{\C}{\mathbb{C}}

\newcommand\wt{\mathfrak{w}}

\newcommand\gt{\mathfrak{g}}

\newcommand\Rad{\mathpzc{R}}
\newcommand\hg{\mathfrak{h}}

\newcommand\spt{\mathfrak{sp}}

\newcommand\Symm{\mathpzc{S}}

\newcommand\qq{\mathpzc{q}}
\newcommand\pq{\mathpzc{p}}

\newcommand\mt{\textswab{m}}

\newcommand\Id{\mathrm{I}}

\newcommand\diag{\mathrm{diag}}

\newcommand\nt{\mathfrak{n}}

\newcommand\gl{\mathfrak{gl}}

\newcommand\slt{\mathfrak{sl}}
\newcommand\su{\mathfrak{su}}

\newcommand\rt{\mathfrak{r}}
\newcommand\Pp{\mathfrak{P}}
\newcommand\td{\mathfrak{t}}

\newcommand\cg{\mathfrak{g}^{c}}

\newcommand\ttt{\texttt{t}}

\newcommand{\lt}{\mathfrak{l}}
\newcommand\G{\mathfrak{G}}

\newcommand\vq{\mathpzc{v}}
\newcommand\wq{\mathpzc{w}}

\newcommand\bil{\textswab{b}}

\newcommand\Hb{\mathbb{H}}

\newcommand\Der{\mathpzc{Der}}

\newcommand\aq{\mathpzc{a}}
\newcommand\hq{\mathpzc{h}}

\newcommand\Wi{\mathpzc{W}}

\newcommand\eq{\mathpzc{e}}

\renewcommand\rq{\mathpzc{r}}

\newcommand\Xx{\mathpzc{X}}
\newcommand\epi{\epsilonup}

\newcommand\sq{\mathpzc{s}}
\newcommand\cq{\mathpzc{c}}

\newcommand\K{\mathbb{K}}

\newcommand\Aa{\mathbb{A}}

\newcommand\pes{\mathpzc{W}}
\newcommand\co{\mathfrak{co}}
\newcommand\kq{\mathpzc{k}}
\newcommand\sfk{\textsf{k}}
\newcommand\sfV{\textsf{V}}
\newcommand\sfW{\textsf{W}}
\newcommand\sfJ{\textsf{J}}

\newcommand\dgt{\Tilde{\mathfrak{g}}}

\newcommand\glc{\mathfrak{gl}^{\vee}}

   \def\DHLhksqrt#1#2{\setbox0=\hbox{$#1\sqrt{#2\,}$}\dimen0=\ht0
     \advance\dimen0-0.2\ht0
     \setbox2=\hbox{\vrule height\ht0 depth -\dimen0}%
     {\box0\lower0.4pt\box2}}

\title[$\Z$-graded Lie algebras]{On some classes of 
$\Z$-graded Lie algebras}
\author{S.Marini, C.Medori, M.Nacinovich}

\address{Stefano Marini: Dipartimento di Scienze Matematiche, 
Fisiche e Informatiche\\ Universit\`a di Parma\\ Parco Area delle Scienze 
53/A (Campus), 43124 Parma
 (Italy)} \email{stefano.marini@unipr.it}

\address{Costantino Medori:
Dipartimento di Scienze Matematiche, Fisiche e Informatiche\\ 
Universit\`a di Parma\\ Parco Area delle Scienze 53/A (Campus), 
43124 Parma
 (Italy)} \email{costantino.medori@unipr.it}

\address{Mauro Nacinovich:
Dipartimento di Matematica\\ II Universit\`a di Roma
``Tor Ver\-ga\-ta''\\ Via della Ricerca Scientifica\\ 
00133 Roma
(Italy)}
\email{nacinovi@mat.uniroma2.it}


\begin{document}
\maketitle
\begin{abstract}
We study finite dimensional 
\emph{almost} and \emph{quasi-effective} 
prolongations
of nilpotent $\Z$-graded 
Lie algebras, especially focusing on those
having a decomposable
reductive structural subalgebra.
Our  assumptions
generalize  \emph{effectiveness} and \emph{algebraicity}
and are appropriate to obtain
Levi-Mal\v{c}ev and Levi-Chevalley decompositions
and precisions on the heigth and other properties of
the prolongations in a very natural way.
In a last section we systematically present
examples in which  
simple Lie algebras are 
obtained as 
prolongations, for 
reductive structural algebras of type A, B, C and D, of
nilpotent $\Z$-graded Lie algebras  arising as 
their linear representations.
\end{abstract}
\tableofcontents

\section*{Introduction}
Cartan's method for the study of equivalence and symmetries of
differential $\Gf$-structures naturally leads
 to consider
$\Z$-graded Lie algebras and their prolongations. 
This approach is clearly explained in 
 the classical books  \cite{Sternberg} of Sternberg
 and \cite{Kob} of Kobayashi. The work of 
 N. Tanaka (see e.g. \cite{Tan67, Tan70}),
extending the scope to general 
contact and $CR$ structures, 
 set the path 
 for further
developments of the subject  
(see e.g. \cite{AlSp,MMN2018,Ottazzi2011}). 
An additional motivation is the fact that
filtered Lie algebras are the core of the
algebraic model for 
transitive differential geometry (see \cite{GS}).
Their $\Z$-graded associated  objects 
have therefore an essential role in the study of 
several
differential geometrical structures.
Our interest 
in this topic was fostered
by our previous work on homogeneous
$CR$ manifolds 
(se e.g. \cite{AMN06, AMN06b, AMN10b, AMN2013, LN05, LN08,
MaNa1, MaNa2}).
The more recent \cite{NMSM} showed us 
that  
some of the $\Z$-graded Lie algebra 
naturally arising in this context 
do not satisfy all standard requirements 
of Tanaka's theory under which e.g. are usually discussed 
effective maximal prolongations 
(cf. \cite{Ottazzi2011,Warhurst2007}). 
This motivates our consideration 
of fairly  general classes of 
$\Z$-graded Lie algebras.  
Our main concern here is not on the maximality of prolongations,
for which we refer throughout 
to our \cite{MMN2018}.
Instead, we often restrict to  prolongations that are 
\textit{assumed}
to be 
finite dimensional. 
Since in this case
homogeneous terms of degree different from zero
are $\ad$-nilpotent, we keep most of the structure
of a $\Z$-graded Lie algebra 
while 
passing, following \cite[Ch.VII,\S{5}]{Bou82}, 
to its \emph{decomposable envelope}.
This observation was very useful to clarify and simplify
several points of the theory.
We found also convenient to introduce 
weaker assumptions of
effectiveness which are nevertheless sufficient to
get informations on the positive degree homogeneous 
summands.
Let us briefly summarise the contents of the paper.
\par 
In the first section we collect some general
structure property of $\Z$-graded Lie algebras. 
As explained above, a 
key role is played by the concept of decomposability,
which generalizes
algebraicity and which is used 
here also
to obtain a shorter proof
of the existence of $\Z$-graded Levi-Mal\v{c}ev and
Levi-Chevalley decompositions
(cf. \cite{CC2017,MN02}). \par 
In \S\ref{sec-effect} we consider various effectiveness conditions,
weakening 
those in \cite{Tan67}, 
and some of their consequences.\par 
We call the subalgebra $\gt_{0}$ 
of $0$-degree homogeneous terms of a $\Z$-graded Lie algebra
$\gt$ its \textit{type},
or \textit{structural subalgebra}. We look at $\gt_{0}$ as  
infinitesimally 
describing the basic 
symmetries of the structure of a differential geometrical object under
consideration. Starting from
\S\ref{sec-reduc}, we study the consequences
of assuming that the adjoint action of $\gt_{0}$ on $\gt$ is reductive.
\par  
In \S\ref{sec-quasi} we come back to the effectiveness conditions,
showing that, together with assumptions on the type,  
they 
bring some further 
precision on the 
features of the $\Z$-graded
Levi-Mal\v{c}ev decomposition.\par 
In the last \S\ref{sect9} we deal with semisimple
prolongations. After some general considerations, 
we systematically develop a series
of examples, also
relating to the exceptional Lie algebras and in which
spin representations play an important role. 
We believe that some of them could also be of some
interest in physics, where graded Lie algebras 
may contribute to better understand  basic symmetries
of nature. For the nilpotent depth $2$ case some similar
descriptions were obtained in \cite{Mor2018}.
\section{$\Z$-graded Lie alebras}\label{struct}
In this preliminary section we fix some notation 
and definitions 
that will be used throughout the paper. 
We consider $\Z$-graded Lie algebras 
\begin{equation}\label{e8.1}
 \gt\,{=}\,{\sum}_{\pq\in\Z}\gt_{\pq}
\end{equation}
over a field $\K$ of
characteristic $0.$  
\begin{defn}
 The Lie subalgebra $\gt_{0}$ of the $\Z$-graded Lie algebra \eqref{e8.1} is called
 its \emph{structure subalgebra}. The \emph{depth} of $\gt$ is $\muup{=}\sup\{\pq\,{\in}\,\Z\,{\mid}\,
 \gt_{-\pq}{\neq}0\}$ end its \emph{heigth} is $\nuup{=}\sup\{\pq\,{\in}\,\Z\,{\mid}\,
 \gt_{\pq}{\neq}0\}.$
\end{defn}
\par 
The map 
\begin{equation} \label{char}
 D_{E}:\gt\to\gt,\;\;\text{with}\;\; D_{E}(X_{\pq})
 =\pq\cdot{X}_{\pq},\;\;\forall\pq\in\Z,\;\;\forall
 X_{\pq}\in\gt_{\pq}
\end{equation}
is a degree $0$ derivation of $\gt,$ that we call \emph{characteristic}.
\begin{defn}
 We call \emph{characteristic} 
 a $\Z$-graded Lie algebra $\gt$ which contains
 a homogeneous element  $E$ of degree $0$ 
 (its  \emph{characteristic element}) 
 such that $[E,X]=D_{E}(X)$ for all $X\,{\in}\,\gt.$
 \end{defn}
\begin{defn}
 A $\Z$-graded Lie algebra \, $\gt'{=}{\sum}_{{\pq}{\in}\Z}\gt'_{\pq}$\,  is called a
 \emph{prolongation} of the $\Z$-graded Lie algebra $\gt$ of \eqref{e8.1} if 
 $\gt$ is a Lie subalgebra of $\gt'$ and 
\begin{equation}
 \gt'_{\pq}=\gt_{\pq},\;\forall\pq\,{<}\,0,\;\;\; \gt'_{\pq}\supseteqq\gt_{\pq},\;\forall\pq\,{\geq}\,0.
\end{equation}
We say that a $\Z$-graded prolongation $\gt'$
is \emph{of type $\gt_{0}$} if $\gt'_{0}\,{=}\,\gt_{0}.$  
\end{defn}
 For a thorough discussion of 
prolongations of $\Z$-graded Lie algebras we refer the reader to \cite{MMN2018}. 
\begin{lem}\label{lm8.1} A $\Z$-graded Lie algebra $\gt$ admits a 
prolongation $\gt\,{\hookrightarrow}\,\cg,$
with $\cg{=}\gt$ if $\gt$ is characteristic and $\dim_{\K}(\cg/\gt)\,{=}\,1$
otherwise. 
The adjoint representation of $\cg$ restricts to a
representation of $\gt$ on $\cg,$ whose kernel 
is the centralizer of $\gt$ in $\gt_{0}$:
\begin{equation}\label{8-3eq}
\ct_0=\{A\in\gt_0\mid [A,\gt]\,{=}\,\{0\}\}.
\end{equation}
\end{lem}\begin{proof}
If $\gt$ is not characteristic, we may consider 
$\cg=\gt\,{\oplus}\,\langle{E}\rangle$, with the gradation $\cg{=}{\sum}_{\pq\in\Z}\cg_{\pq}$
in which $\cg_{\pq}{=}\gt_{\pq}$ when $\pq{\neq}0,$ 
$\cg_{0}\,{=}\,\gt_{0}\,{\oplus}\,\langle{E}\rangle$ and 
the Lie algebra structure on $\cg$ is defined by requiring that $\gt$ is an ideal in $\cg$ and
that
$[E,X]\,{=}\,D_{E}(X)$ for all $Z\,{\in}\,\gt.$ \par 
Formula \eqref{8-3eq} follows because 
the elements of 
the center of a characteristic $\Z$-graded Lie algebra are homogeneous of degree zero.\end{proof}
\begin{defn} The $\Z$-graded Lie algebra $\cg$ 
of Lemma~\ref{lm8.1} will be called the
\emph{characteristic prolongation} of $\gt.$ 
\end{defn}

\subsection{Finite dimensional $\Z$-graded Lie algebras}
\label{secnot}
Assuming that  
$\gt$ is 
finite dimensional, 
we denote by $\rt$ its solvable radical  
and by $\nt$ its maximal nilpotent ideal.
Both $\rt$ and $\nt$ are graded ideals of $\gt$: 
\begin{equation}\label{8.2}
 \nt={\sum}_{\pq\in\Z}\nt_{\pq} \subseteq\rt={\sum}_{\pq\in\Z}\rt_{\pq}
\end{equation}
and $\nt$ consists of the $\ad_{\gt}$-nilpotent elements of $\rt.$ 
In particular, 
\begin{equation} \label{8.3}
\begin{cases}
  \nt_{\pq}=\rt_{\pq}, & \forall \pq\neq{0},\\
  \nt_{0}\subseteq\rt_{0}
\end{cases}
\end{equation}
and any derivation $D$ of $\gt$ maps $\rt$ into $\nt.$ 
Set 
\begin{equation}
 \mt={\sum}_{\pq<0}\gt_{\pq},
 \quad \gt_{+}={\sum}_{\pq{\geq}0}\gt_{\pq},
\end{equation}
and denote by $\sfk_{\gt}$ the Killing form of $\gt.$
We recall, from \cite[Ch.I,\S{5}.5,Prop.5]{n1998lie}, 
that
 the radical $\rt$  is the orthogonal of $[\gt,\gt]$ 
 with respect to~$\sfk_{\gt}.$ 
\subsection{$\Z$-graded linear representations}
 A $\Z$-graded $\K$-vector space
is the datum of a $\K$-vector space $\sfV$ and  of its decomposition 
\begin{equation}\label{eq8.12}
\sfV={\sum}_{\pq\in\Z}\sfV_{\pq},
\end{equation}
into a direct sum of vector subspaces, indexed by the integers. 
A linear map
$\phiup\,{\in}\,\gl_{\K}(\sfV)$ is 
called \emph{homogeneous of degree $\pq$}  if 
\begin{equation*}
 \phiup(V_{\qq})\subseteq\sfV_{\qq+\pq},\;\;\forall\qq\in\Z.
\end{equation*}
The homogeneous elements of $\gl_{\K}(\sfV)$ generate a Lie subalgebra
$\glc_{\K}(\sfV)$ of $\gl_{\K}(\sfV),$ which has the natural 
$\Z$-gradation 
\begin{equation}\label{eqgrad}
\begin{cases}
 \glc_{\K}(\sfV)={\sum}_{{\pq\in\Z}}[\glc_{\K}(\sfV)]_{\pq}, 
 \;\;\text{with}\;\;\\
[\glc_{\K}(\sfV)]_{\pq}=\{A\in\gl_{\K}(\sfV)\mid A(\sfV_{\qq})
\subseteq\sfV_{\qq{+}\pq},\;\forall
\qq\in\Z\}.
\end{cases}\end{equation}
If the $\Z$-gradation of $\sfV$ is finite, i.e. 
 when the
set of $\pq\,{\in}\Z$ with
$\sfV_{\pq}\,{\neq}\,\{0\}$ is finite, 
then $\glc_{\K}(\sfV)\,{=}\,\gl_{\K}(\sfV).$ 
We keep however  also in this case
the notation $\glc_{\K}(\sfV)$ to specify that we are considering  on 
$\gl_{\K}(\sfV)$ 
the 
$\Z$-gradation \eqref{eqgrad}, related to a given $\Z$-gradation 
\eqref{eq8.12} of $\sfV.$ 
We denote by $E_{\sfV}$ the map in $\glc_{\K}(\sfV)$ with
$E_{\sfV}(\vq_{\pq})=\pq\,{\cdot}\,\vq_{\pq}$ for $\pq\,{\in}\Z$ and 
$\vq_{\pq}\,{\in}\,\sfV_{\pq}.$ 
\begin{defn} A $\Z$-graded 
\emph{linear representation} of $\gt$ is 
a linear representation 
$\rhoup\,{:}\,\gt\,{\to}\,\glc_{\K}(\sfV)$ 
of $\gt$ on a $\Z$-graded vector space 
$\sfV,$
such that 
\begin{equation} \label{qq8.11}
\rhoup(X_\pq)(\sfV_{\qq})\subseteq\sfV_{\pq+\qq},
\;\;\;\forall \pq,\qq\in\Z,\;X_{\pq}\in\gt_{\pq}.
\end{equation}
\end{defn}
\begin{prop}\label{propfinrep}
 Let $\gt$ be a $\Z$-graded characteristic Lie algebra. If 
 $\sfV$ is a finite dimensional $\K$-vector space and 
 $\rhoup\,{:}\,\gt\,{\to}\,\gl_{\K}(\sfV)$ 
 a linear Lie algebra representation of $\gt,$
 then we can find a gradation \eqref{eq8.12} of $\sfV$ for which 
 $\rhoup$
 satisfies \eqref{qq8.11}.
\end{prop} 
\begin{proof} Let $E$ be the characteristic element of $\gt.$ 
For polynomials $\psiup(\ttt)\,{\in}\,\K[\ttt],$ we set 
\begin{equation*}
 V_{\psiup(\ttt)}=\ker(\psiup(\rhoup(E)))=
 \{\vq\in\sfV\mid \psiup(\rhoup(E))(\vq)=0\}.
\end{equation*} 
Take the spectral decomposition of $\sfV$ with respect to $\rhoup(E)$:
\begin{equation*}
 \sfV=\sfV_{\psiup_{1}(\ttt)}\oplus\cdots\oplus\sfV_{\psiup_{m}(\ttt)}.
\end{equation*}
Here
 $\psiup_{1}(\ttt),\hdots,\psiup_{m}(\ttt)$ are powers of distinct irreducible
 monic 
polynomials in $\K[\ttt].$ 
For every $\pq\,{\in}\,\Z$ and
$X_{\pq}\,{\in}\,\gt_{\pq},$ 
we have $\rhoup(X_{\pq})\,{\circ}(\rhoup(E){+}\pq{\cdot}\Id_{\sfV})
=\rhoup(E)\,{\circ}\,\rhoup(X_{\pq})$ and hence, for every positive integer $k,$ 
\begin{equation*} \rhoup(X_{\pq})
\circ(\rhoup(E)\,{+}\,\pq{\cdot}\Id_{V})^{k}
=\rhoup(E)^{k}\circ\rhoup(X_{\pq}).
 \end{equation*}
 This shows that 
\begin{equation*}
 \rhoup(X_{\pq})(\sfV_{\psiup(\ttt)})
 \subseteq\sfV_{\psiup(\ttt+\pq)},\;\;\forall\psiup(\ttt)\in\K[\ttt].
\end{equation*} 
Fix $\phiup_{1}(\ttt),\hdots,\phiup_{r}(\ttt)\in
\{\psiup_{1}(\ttt),\hdots,\psiup_{m}(\ttt)\}$ such that 
\begin{equation*} 
\begin{cases}
 \phiup_{i}(\ttt)\neq\phiup_{j}(\ttt{+}\pq),\;\;\forall j\,{\neq}\,i,\;\forall\pq\,{\in}\,\Z,\\
 \{\psiup_{1},\hdots,\psiup_{m}\}
 \subset\{\phiup_{i}(\ttt\,{+}\,\pq)
 \mid 1{\leq}i{\leq}\rq,\; \pq\in\Z\}.
\end{cases}
\end{equation*} 
Then the gradation 
\begin{equation*}
 \sfV={\sum}_{\pq\in\Z}\sfV_{\pq},\;\;\text{with}\;\; 
 \sfV_{\pq}={\sum}_{i=1}^{r}\sfV_{\phi_{i}(\ttt+\pq)}
\end{equation*}
satisfies the requirements of the Theorem. 
\end{proof}
\begin{rmk}
 Any $\Z$-graded linear representation of a $\Z$-graded Lie algebra 
 $\gt$ extends to
 a linear representation of its characteristic prolongation $\cg.$ 
 Thus Prop.~\ref{propfinrep}
 tells us that actually all finite dimensional 
 $\Z$-graded linear representations of $\gt$
 are restrictions to $\gt$ of linear representations of $\cg.$ 
 The gradation on $\sfV$ is not
 uniquely determined: indeed we can always change 
its grading
 by shifting the indices
 by an integral constant and this can be done independently 
 on each subspace 
 $V^{i}={\sum}_{\pq\in\Z}V_{\phiup_{i}(\ttt+\pq)}.$ Note also that 
 $\rhoup(E)$ may not be
 a characteristic element of~$\glc_{\K}(\sfV).$ In the finite dimensional case we can normalise 
 the grading of $\sfV$ by requiring that $E_{\sfV}\,{\in}\,\slt_{\K}(\sfV),$ i.e. that
 ${\sum}_{\pq\in\Z}\,(\pq\,{\cdot}\dim_{\K}(\sfV_{\pq}))\,{=}\,0.$ 
\end{rmk}
We obtain a graded version of Ado's theorem.
\begin{thm}\label{ado}
 Every finite dimensional $\Z$-graded Lie algebra admits a $\Z$-graded
 faithful representation such that the linear maps corresponding to 
 elements of its maximal nilpotent ideal  are nilpotent. 
\end{thm}\begin{proof} 
Let $\gt$ be a finite dimensional $\Z$-graded Lie algebra
and $\cg$ its characteristic prolongation.
By Ado's theorem (see e.g. \cite[Ch.I,\S{7.3},Theorem~3]{n1998lie}) 
$\cg$ has 
a faithful 
finite dimensional linear representation 
$\rhoup{:}\cg{\to}\gl_{\K}(\sfV)$
such that 
$\rhoup(X)$ is nilpotent for all $X$ 
in the maximal nilpotent ideal of $\cg,$
which coincides with the maximal nilpotent ideal $\nt$
of $\gt.$ By Prop.~\ref{propfinrep} we can find a $\Z$-gradation of $\sfV$
to make $\rhoup$ graded.\end{proof}
\begin{rmk} 
 If $\ct_{0}\,{=}\,0,$ i.e. if no nonzero element of the center 
 is homogeneous of degree $0,$ 
 then the restriction to $\gt$ of the adjoint representation 
 of its characteristic prolongation 
 $\cg$ is a faithful $\Z$-graded
 finite dimensional linear representation 
 of $\gt,$ mapping the elements of $\nt$
 into nilpotent endomorphisms. 
\end{rmk}
\begin{defn} Let $\gt$ be a finite dimensional $\Z$-graded 
Lie algebra 
and $\sfV$ a $\Z$-graded finite dimensional vector space over 
$\K.$ 
 A finite dimensional \mbox{$\Z$-graded} faithful representation 
 $\rhoup\,{:}\,\gt\,{\to}\,\glc_{\K}(\sfV)$ 
 for which $\rhoup(\nt)$ is a Lie algebra 
 of nilpotent endomorphisms of $\sfV$ will be called
 a \emph{realisation} of $\gt.$  
\end{defn}
\subsection{Decomposable Lie algebras} 
In this subsection we will not assume that the Lie algebras we consider are $\Z$-graded.\par
Let $\sfV$ be a finite dimensional vector 
space over a field $\K$ of characteristic $0.$ 
Every
$A$ in  $\gl_{\K}(\sfV)$ admits a Jordan-Chevalley decomposition:
there are $A_{s},A_{n}\,{\in}\,\gl_{\K}(\sfV)$ such that
\begin{align*}
A=A_{s}+A_{n},\;[A_{s},A_{n}]=0,\;
\; \text{with $A_{s}$ semisimple and
$A_{n}$ nilpotent on $\sfV.$}
\end{align*}
The summands $A_{s},$   $A_{n}$ are uniquely determined and
$A_{s}\,{=}\,\psiup_{s}(A),$ $A_{n}\,{=}\,\psiup_{n}(A)$ with
$\psiup_{s}(\ttt),\psiup_{n}(\ttt)\,{\in}\,\K[\ttt]$ 
polynomials with
no constant term. \par 
We recall some notion and results from
\cite[Ch.VII,\S{5}]{Bou82}. \par 
The \emph{decomposable envelope} $\dgt$ 
of a Lie subalgebra
$\gt$ of  $\gl_{\K}(\sfV)$ is the Lie subalgebra of $\gl_{\K}(\sfV)$ generated by
the semisimple and nilpotent components  of its elements. Note that 
$\gt$ is an ideal in $\dgt.$ We say that $\gt$ is
\emph{decomposable in $\sfV$}
when $\dgt\,{=}\,\gt.$  A necessary and sufficient condition for $\gt$ to
be decomposable in $\sfV$ is that its solvable radical $\rt$ is decomposable in $\sfV.$ 
\par 
We denote by
\begin{equation}
 \nt_{\sfV}=\{X\in\rt\mid X\;\text{is nilpotent on $\sfV$}\}
\end{equation}
the maximal ideal in $\gt$ consisting of nilpotent endomorphisms of $\sfV.$ We have 
\begin{equation*}
 [\gt,\gt]\cap\rt\subseteq\nt_{\sfV}\subseteq\nt.
\end{equation*} 
We say that $\gt$ is \emph{$\sfV$-reductive} if the $\gt$-module 
$\sfV$ is
completely reducible. A $\sfV$-reductive $\gt$ is also
decomposable in $\sfV.$ \par 
\begin{defn}
 Let $\gt$ be a Lie subalgebra of $\gl_{\K}(\sfV).$ A \emph{Levi-Chevalley decomposition} of
 $\gt$ in $\sfV$ is a direct sum decomposition 
\begin{equation}\label{LCdec}
 \gt=\lt_{\sfV}\oplus\nt_{\sfV},
\end{equation}
where $\lt_{\sfV}$ is a $\sfV$-reductive Lie subagebra of $\gt.$ \par
 We call such an $\lt_{\sfV}$ 
a \mbox{\emph{$\sfV$-reductive Levi factor}} of $\gt.$ 
\end{defn}
The condition of being decomposable in $\sfV$ 
is necessary and sufficient to ensure that $\gt$ admits
a Levi-Chevalley decomposition in $\sfV$ and the group of elementary
automorphisms of $\gt$ is transitive on the set of $\sfV$-reductive Levi factors of $\gt.$
This is the contents of \cite[Ch.VII,\S{5}, Proposition~7]{Bou82}, that we further precise
by giving here the  
analogue  of a theorem that Mostow (\cite{Most56}) proved for algebraic groups.
\begin{prop}\label{prop8.7}
 Let $\gt$ be a decomposable Lie subalgebra of $\gl_{\K}(\sfV)$ and 
 $\td$ a commutative Lie subalgebra of $\gt,$ consisting of 
 semisimple endomorphisms. 
 Then $\gt$ has a $\sfV$-reductive Levi factor 
 $\lt_{\sfV}$ containing $\td.$ 
\end{prop}
\begin{proof} Take a $\sfV$-reductive Levi factor  $\lt_{\sfV}$ of $\gt$ 
for which $\lt_{\sfV}\,{\cap}\,\td$ is maximal. We claim that $\td\,{\subseteq}\,\lt_{\sfV}.$ 
To prove this fact we argue
by contradiction. Assume that there is an $A\,{\in}\,\td{\backslash}\lt_{\sfV}.$
This $A$ uniquely decomposes into a sum $A\,{=}\,A'{+}N,$ with $A'\,{\in}\lt_{\sfV}$
and $N\,{\in}\,\nt_{\sfV}.$ If $B\,{\in}\,\td\,{\cap}\lt_{\sfV},$ from 
\begin{equation*}
 0=[A,B]=[A',B]+[N,B]
\end{equation*}
we obtain that both $[A',B]\,{=}\,0$ and $[N,B]\,{=}\,0,$ because the first summand
is in $\lt_{\sfV}$ and the second in $\nt_{\sfV}.$ Let $X$ and $Y$ be the semisimple
and nilpotent summands in the Jordan-Chevalley decomposition
of $A'.$ Being polynomials in $A',$ they both commute with the elements of $\td\,{\cap}\,\lt_{\sfV}.$ 
Let us consider now the Lie subalgebra $\kt$ of $\gl_{\K}(\sfV)$ generated by
$X,Y,N.$ It is solvable and the ideal 
$\nt_{\sfV}'$ of its nilpotent endomorphisms, which is
generated by $Y,N,[X,N],$ has codimension $1$ in $\kt$ and its elements commute
with those in $\td\,{\cap}\,\lt_{\sfV}.$ We note that both $\langle{X}\rangle$
and $\langle{A}\rangle$ are $\sfV$-reductive Levi factors of $\kt.$ Then
there is an elementary automorphism $\Psi$
of $\kt$ mapping $\langle{X}\rangle$ onto
$\langle{A}\rangle.$ This $\Psi$ is a composition of 
automorphisms of the form
$\exp(\ad_{\kt}(T)),$ with $T\,{\in}\,\nt_{\sfV}'.$ Its extension to an automorphism $\tilde{\Psi}$
of $\gt$ is a composition of  $\exp(\ad_{\gt}(T)),$ with $T\,{\in}\,\nt_{\sfV}'$ and hence
leaves invariant all elements of $\td{\cap}\lt_{\sfV}.$ Therefore $\tilde{\Psi}(\lt_{\sfV})$ is
a $\sfV$-reductive Levi factor with 
\begin{equation*}
 \td\cap\lt_{\sfV}\oplus\langle{A}\rangle\subseteq\tilde{\Psi}(\lt_{\sfV}),
\end{equation*}
contradicting the choice of $\lt_{\sfV}$ and showing therefore that in fact 
$\td\,{\subseteq}\,\lt_{\sfV}.$ The proof is complete. 
\end{proof}

\subsection{Decomposable prolongations} We specialise the general 
notions of the previous subsection
to the case of $\Z$-graded Lie algebras.
\begin{prop}
Let $\sfV$ be a finite dimensional $\Z$-graded $\K$-vector space and
$\gt$ a $\Z$-graded Lie subalgebra of $\gl_{\K}(\sfV).$ 
 Then its decomposable envelope $\dgt$ in $\sfV$ is the $\Z$-graded 
 Lie subalgebra 
\begin{equation}\label{e8.16}
 \dgt={\sum}_{\pq\in\Z}\dgt_{\pq}
\end{equation}
 of $\gl_{\K}(\sfV),$ where
 $\dgt_{\pq}\,{=}\,\gt_{\pq}$ 
 for all $\pq{\neq}0,$
while $\dgt_{0}$ equals  the decomposable envelope of $\gt_{0}$ 
 in $\sfV.$ 
\end{prop} 
\begin{proof}
The Lie algebra defined by \eqref{e8.16} 
is decomposable by \cite[Ch.VII, \S{5.5}, Theorem~1]{Bou82},
being generated by the elements of 
${\bigcup}_{\pq{\neq}0}\gt_{\pq},$ which are nilpotent,
and by the semisimple and nilpotent components of the elements of 
$\gt_{0}.$ Moreover, $\gl_{\K}(\sfV)$ contains a characteristic
element $E_{\sfV}$ and 
\begin{equation*}
[\gl_{\K}(\sfV)]_{0}\,{=}\,\{A\,{\in}\,\gl_{\K}(\sfV)
\,{\mid}\,E_{\sfV}{\circ}\,A\,{=}\,A\,{\circ}\,E_{\sfV}\}.
\end{equation*}
Since the semisimple and nilpotent summands $A_{s}$ and $A_{n}$ 
of an element 
$A$ of $\gt_{0}$ are polynomials of $A,$ 
they also commute with $E_{\sfV}$ and therefore belong to $[\gl_{\K}(\sfV)]_{0}.$ 
\end{proof}

\subsection{$\Z$-graded Levi-Mal\v{c}ev 
and Levi-Chevalley decompositions}
Finite dimensional   $\Z$-graded 
Lie algebras admit $\Z$-graded
Levi-Mal\v{c}ev and Levi-Chevalley
decompositions. 
The Levi-Mal\v{c}ev decompositions was
stated and proved
in  \cite{CC2017,MN02} for real and complex
finite dimensional graded Lie algebras.
 We provide here
a short proof for general fields of characteristic zero. 
\begin{thm}\label{lmc-dec}
 Every finite dimensional  $\Z$-graded 
 Lie algebra 
 admits a $\Z$-graded 
 Levi-Mal\v{c}ev decomposition.\par
 Let $\sfV$ be a finite dimensional $\Z$-graded $\K$-vector space. Then
 every decomposable $\Z$-graded Lie subalgebra of $\gl_{\K}(\sfV)$
 admits a Levi-Chevalley decomposition in $\sfV.$ 
\end{thm} 
\begin{proof} We begin by proving the last statement under an additional assumption.
Let $\sfV$ be a finite dimensional $\Z$-graded vector space and 
assume that $\gt$ is a decomposable $\Z$-graded
Lie subalgebra of $\gl_{\K}(V),$ 
containing the characteristic element $E_{\sfV}$ 
 of $\gl_{\K}(\sfV).$
Since $E_{\sfV}$ is semisimple, 
by Proposition~\ref{prop8.7}, there is a 
$\sfV$-reductive Levi factor $\lt_{\sfV}$ of $\gt$ containing 
$E_{\sfV}$ and hence $\Z$-graded.   
\par
Consider now any $\Z$-graded Lie algebra $\gt.$ 
By using Theorem~\ref{ado},
we  identify $\gt$ with
a $\Z$-graded Lie subalgebra of $\gl_{\K}(\sfV),$ for a suitable finite dimensional
$\Z$-graded  $\K$-vector space $\sfV.$ Then $\cg{\coloneqq}\gt{+}\langle{E_{\sfV}}\rangle$
is a $\Z$-graded Lie subalgebra of $\gl_{\K}(\sfV)$ which, by the argument above,
has a $\Z$-graded $\sfV$-reductive Levi factor $\lt_{\sfV}.$ Then
$\ses{\coloneqq}[\lt_{\sfV},\lt_{\sfV}]$ is a semisimple 
$\Z$-graded Levi factor of $\gt$.\par
It remains to prove the existence a $\Z$-graded Levi-Chevalley decomposition in $\sfV$
for a decomposable $\Z$-graded Lie subalgebra $\gt$ of $\gl_{\K}(\sfV)$ 
without assuming that $E_{\sfV}\,{\in}\,\gt.$  
We begin by taking a Levi-Mal\v{c}ev decomposition
$\gt\,{=}\,\ses\,{\oplus}\,\rt$ with a $\Z$-graded semisimple Levi factor $\ses.$ Since 
${\sum}_{\pq\neq{0}}\rt_{\pq}\,{\subseteq}\,\nt_{\sfV},$ 
we obtain that $\rt\,{=}\,\nt_{\sfV}{+}\rt_{0}.$ The subalgebra $\rt_{0}$ is solvable and
decomposable in $\sfV.$ Therefore, if $\td_{0}$ is a maximal abelian subalgebra of $\rt_{0}$
consisting of 
semisimple endomorphisms of $\sfV,$ then $\rt_{0}\,{=}\,\td_{0}\,{\oplus}\,(\nt_{\sfV}{\cap}\rt_{0})$
(see e.g. \cite[Ch.VII,\S{5}, Corollary~2]{Bou82}) 
and hence $\lt_{\sfV}\,{=}\,\ses\,{\oplus}\,\td_{0}$ is a $\sfV$-reductive Levi factor of $\gt.$
\end{proof}
We recall (see e.g. \cite[Ch.I,\S{6.4}]{n1998lie}) 
that a Lie algebra $\kt$ is \emph{reductive} 
if its adjoint representation is semisimple.
This is equivalent to the fact that its derived algebra 
$[\kt,\kt]$ is semisimple. 
If $\kt$ is reductive, then
its 
center is a direct sum complement of 
 $[\kt,\kt]$  in $\kt.$\par 
A 
finite dimensional Lie algebra $\gt$ is \emph{decomposable} if 
the semisimple and nilpotent components of its inner derivations are still
inner derivations: this means that 
$\ad_{\gt}(\gt)$
is a $\gt$-decomposable Lie subalgebra of $\gl_{\K}(\gt).$ \par 
A \textit{linear realisation} of $\gt$ is 
a faithful finite dimensional linear representation
$\rhoup\,{:}\,\gt\,{\to}\,\gl_{\K}(\sfV)$ such that $\rhoup(X)$ is nilpotent on $\sfV$ for each
$X$ in the maximal nilpotent ideal $\nt$ of $\gt.$ If $\gt$ is decomposable, by using its
linear realisation, we obtain a direct sum decomposition 
\begin{equation}
 \gt\,{=}\,\lt\,{\oplus}\,\nt,
\end{equation}
where $\lt$ is a reductive Lie algebra of $\gt.$ We call 
\emph{reductive Levi factors} of $\gt$
the reductive Lie algebras which are
complements of $\nt$ in $\gt.$  The elementary automorphisms
act transitively on the set of reductive Levi factors of $\gt.$
\par 
From Theorems~\ref{ado} and \ref{lmc-dec}
we obtain 
\begin{thm}
A decomposable finite dimensional $\Z$-graded Lie algebra $\gt$
contains $\Z$-graded reductive factors.\qed
\end{thm}

\begin{lem} \label{lem8.2}
Let 
$\gt\,{=}\,{\sum}_{\pq\in\Z}\gt_{\pq}$  
be a characteristic
finite dimensional $\Z$-graded Lie algebra.  
Then 
the Cartan subalgebras of $\gt_{0}$ 
are Cartan subalgebras of~$\gt.$   
\end{lem} 
\begin{proof} Every Cartan subalgebra $\hg$ of $\gt_{0}$ 
contains the characteristic element
$E,$ because it belongs to
 the center of $\gt_{0}.$ Let $\hg$ be a Cartan subalgebra of 
$\gt_{0}.$ If  
 $X\,{=}\,{\sum}_{\pq\in\Z}X_{\pq},$ 
 with $X_{\pq}\,{\in}\,\gt_{\pq},$ is in
the normaliser of $\hg$ in $\gt,$ then the condition  
\begin{equation*}
 [E,X]={\sum}_{\pq\in\Z}\pq\cdot{X}_{\pq}\in\hg\subseteq
 \gt_{0}
\end{equation*}
implies that $X$ belongs to 
$\gt_{0}$ and therefore to $\hg,$ 
which is its own normaliser
in~$\gt_{0}.$ Thus $\hg$ is also its own 
normaliser in $\gt$ and thus  is Cartan 
in~$\gt.$ 
\end{proof}
 In a $\Z$-graded finite dimensional Lie algebra $\gt{=}{\sum}_{\pq\in\Z}\gt_{\pq},$
 when $\pq{+}\qq{\neq}0$
the subspaces $\gt_{\pq}$ and $\gt_{\qq}$ are
orthogonal for the Killing form $\sfk_{\gt}.$ 
 Since the restriction of $\sfk_{\gt}$ to a semisimple Levi factor of $\gt$ is nondegenerate, 
Theorem~\ref{lmc-dec} yields
\begin{prop}
 Let $\gt$ be a $\Z$-graded finite dimensional Lie algebra. Then: 
\begin{enumerate}
 \item if $\mt\,{\subseteq}\,\rt,$ then ${\sum}_{\pq{\neq}0}\gt_{\pq}\subseteq\nt$; 
 \item if $\mt\,{\cap}\,\rt\,{=}\,\{0\},$ then $\dim_{\K}(\gt_{\pq})=\dim_{\K}(\gt_{-\pq})$
 for all $\pq\,{\in}\,\Z.$ \qed
\end{enumerate}
\end{prop}
 
 \section{Effectiveness conditions}\label{sec-effect}
 Let 
 $\gt\,{=}\,{\sum}_{\pq\in\Z}\gt_{\pq}$
 be a $\Z$-graded Lie algebra and set  
 $\gt_{+}\,{=}\,{\sum}_{\pq{\geq}0}\gt_{\pq},$
 $\mt{=}{\sum}_{\pq<0}\gt_{\pq}.$ 
\begin{defn}[effectiveness conditions] We say that $\gt$ is 
\begin{itemize}
 \item \emph{effective} if 
\begin{equation}
\label{eq8.17}
  \{X\in\gt_{+}\mid [X,\gt_{-1}]=\{0\}\}=\{0\}, 
\end{equation}
\item \emph{quasi-effective} if the following two conditions are fulfilled:
\begin{equation}
 \label{eq8.18}
   \{X\in\gt_{+}\mid [X,\mt]\subseteq\mt\}\subseteq\gt_{0},
  \;\; \{X\in\gt_{0}\mid [X,\mt]=\{0\}\}=\{0\}, 
\end{equation}
\item \emph{almost effective} if 
\begin{equation}
  \label{eq8.19}
  \{X\in\gt_{+}\mid [X,\mt]=\{0\}\}=\{0\}.
\end{equation}
\end{itemize}
 \end{defn}
 We clearly have \eqref{eq8.17}$\Rightarrow$\eqref{eq8.18}$\Rightarrow$\eqref{eq8.19}
 and the three notions are equivalent when $\mt$ is \textit{fundamental}, i.e. generated
 by $\gt_{-1}.$ \par 
 \begin{ntz} \label{ntz-a.e.} 
Let us indicate by $\Pp_{0}(\mt,\gt_{0})$ the collection of all finite dimensional
 $\Z$-graded \textit{almost effective} prolongations of type $\gt_{0}$ of $\mt,$ where
 $\gt_{0}$ is a subalgebra of $\Der_{\!0}(\mt)$ and, for $D\,{\in}\,\gt_{0}$ and $X\in\mt$ we have
 $[D,X]\,{=}\,D(X).$ We denote by $\mt\,{\oplus}\,\gt_{0}$ the object of $\Pp_{0}(\mt,\gt_{0})$
 which is the semidirect product of $\mt$ and $\gt_{0}.$ 
\end{ntz}

\begin{lem} An almost effective $\gt$ contains 
at most one characteristic element.
If $\gt$ is 
almost effective,  
then the restriction to $\gt$ of the adjoint representation of 
its characteristic prolongation 
is faithful.\qed
\end{lem}

\begin{lem}\label{lm8.7}
 Let $\gt$ be an almost effective 
 $\Z$-graded Lie algebra. 
 Then
a nonnegative degree homogeneous derivation of 
 $\gt$  
  is zero
 if and only if its restriction to $\mt$ is zero. 
\end{lem} 
\begin{proof}
 Let $D\,{\in}\,\Der_{\!\qq}(\gt)$ be a homogeneous derivation of degree 
 $\qq{\geq}0$ 
 which vanishes on $\mt.$ We prove by recurrence that 
 $D(\gt_{\pq})\,{=}\,\{0\}$ for all $\pq\,{\in}\,\Z.$  By assumption
this is true  for $\pq{<}0.$ 
Assume that $\pq\,{\geq}\,0$ and $D(\gt_{\sq})\,{=}\,0$
for $\sq{<}\pq.$ 
 If $X_{\pq}\in\gt_{\pq}$ and $Y_{\rq}\,{\in}\,\gt_{\rq}$ 
 with $\rq{<}0,$ then 
\begin{equation*}
[D(X_{\pq}),Y_{\rq}]=D([X_{\pq},Y_{\rq}])-[X_{\pq},D(Y_{\rq})]=0,
\end{equation*}
because the first summand in the right hand side is zero,
being $[X_{\pq},Y_{\rq}]{\in}\gt_{\pq{+}\rq}$ 
and $\pq{+}\rq{<}\pq$; the second
is zero because $D(Y_{\rq})\,{=}\,0$ since $\rq{<}0.$ 
Then $D(X_{\pq})$ is an element of $\gt_{\pq+\qq}$ 
with $[D(X_{\pq}),\mt]\,{=}\,\{0\}$ and hence is $0$
by the almost effectiveness assumption, because $\pq{+}\qq{\geq}0.$
This completes the proof.
\end{proof}
\begin{lem}\label{lemma8.11}
 Let $\gt$ be a finite dimensional almost effective $\Z$-graded Lie algebra.
 Then a degree $0$-homogeneous derivation 
of $\gt$  
 is nilpotent 
 if and only if its restriction to $\mt$ is nilpotent. 
\end{lem} 
\begin{proof} 
 For a derivation $D$ of $\gt$ 
one can easily prove by recurrence that 
\begin{equation*}
 [D^{k}(X),Y]={\sum}_{h=0}^{k}c_{k,h}D^{h}[X,D^{k-h}(Y)],
\end{equation*}
for suitable constants $c_{k,h}.$ Then, if 
$D\,{\in}\,\Der_{0}(\gt)$ is nilpotent on $\mt,$ we can prove recursively that it is nilpotent
 on $\gt.$ Indeed, if $\pq\,{\geq}\,0$ and we know that $D^{m}$ is zero on $\gt_{\qq}$ for
 $\qq{<}\pq,$ then we obtain that, for every $X_{\pq}\in\gt_{\pq}$  
\begin{equation*}
 [D^{2m}(X_{\pq}),Y]={\sum}_{h=0}^{2m}c_{2m,h}D^{h}[X_{\pq},D^{2m-h}(Y)]=0,\;\;\forall Y\in\mt
\end{equation*}
by the inductive assumption, 
 because $[X_{\pq},D^{2m-h}(Y)]\,{\in}\,{\sum}_{\qq<\pq}\gt_{\qq}$ and 
 either $h$ or $2m{-}h$ is ${\geq}m.$ Since  $\gt$ is almost effective,
 this implies that
 $D^{2m}(X_{\pq})\,{=}\,0.$
The proof is complete.
\end{proof}
A consequence of Lemma~\ref{lemma8.11} 
is that for an almost effective $\gt$ 
the Lie subalgebra of the
homogeneous elements of degree $0$ of its maximal nilpotent ideal 
$\nt$ 
only depends on $\mt$ and the structure subalgebra $\gt_{0}.$ 
\begin{ntz}\label{rapp_notation}
Let $\rhoup\,{:}\,\gt_0\,{\to}\,\gl_{\K}(\mt)$ be the linear
representation obtained by restricting the adjoint representation 
to $\mt.$
Likewise, for all $\pq\,{\in}\,\Z,$ 
we denote by
$\rhoup_{\pq}\,{:}\,\gt_0\,{\to}\,\gl_{\K}(\gt_{\pq})$
the restrictions of $\rhoup$ to the homogeneous subspaces
$\gt_{\pq}.$  
\end{ntz}
\begin{thm}\label{thm8.12} Let $\gt_{0}$ be a Lie algebra of degree $0$ derivations
of a finite dimensional $\Z$-graded nilpotent Lie algebra $\mt{=}{\sum}_{\pq<0}\,\gt_{\pq}$.
If $\nt(\gt)$ is the maximal nilpotent ideal of an almost effective prolongation $\gt$
of type $\gt_{0}$ of $\mt,$~then 
\begin{equation}\label{qq8.17} \nt(\gt)\cap\gt_{0}=
 \nt_{0}\coloneqq\{X\in\rt_0
 \mid \rhoup(X)
 \;\;\text{is nilpotent on $\mt$}\}.
\end{equation} 
\begin{proof} Let us fix a $\Z$-graded Levi-Mal\v{c}ev decomposition
$\gt\,{=}\,\ses\,{\oplus}\,\rt$ 
of $\gt.$ The subalgebra 
$\ses_{0}{\coloneqq}\ses\,{\cap}\,\gt_{0}$ of its $\Z$-graded semisimple Levi factor
is reductive and decomposes 
into the direct sum 
$\ses_{0}'\,{\oplus}\,\zt_{0}$ of its semisimple ideal 
$\ses'_{0}{=}\,[\ses_{0},\ses_{0}]$ and
its center $\zt_{0}{=}\,\{X\,{\in}\,
\ses_{0}\,{\mid}\,[X,\ses_{0}]=\{0\}\}.$ 
We note that $\ses$ contains a characteristic element 
$E_{\ses}$ and  therefore 
$\zt_{0}$ is contained in a Cartan subalgebra of $\ses$ contained in 
$\ses_{0}.$ The elements of $\zt_{0}$ are $\ad_{\gt}$-semisimple on $\ses$ and
on $\gt.$ Therefore, 
if $A$ is a nonzero element of $\zt_{0}$ and $B\,{\in}\,\rt,$ 
then  
$\ad_{\gt}(A{+}B)$  
is not nilpotent on $\gt.$
Therefore the $\ad_{\gt}$-nilpotent elements of $\rad(\gt_{0})$ 
are contained in $\rt_{0}$ 
and therefore belong to $\nt_{0}.$ The claim of the Theorem follows  from 
Lemma~\ref{lemma8.11},
because, by the assumption that $\gt$ is almost effective, an $X\,{\in}\,\gt_{0}$
is
$\ad_{\gt}$-nilpotent if and only if 
$\rhoup(X)$ is nilpotent on~$\mt.$ 
\end{proof}
\end{thm}

\begin{lem} \label{lemma8.13} For an almost effective
finite dimensional $\Z$-graded Lie algebra $\gt$ the following are
equivalent: 
\begin{itemize}
 \item[$(i)$] $\gt$ is decomposable;
 \item[$(ii)$] $\ad_{\gt}(\gt_{0})$ is decomposable in $\gt$;
 \item[$(iii)$] $\rhoup(\gt_{0})$ is decomposable in $\mt$.
\end{itemize}
\end{lem} 
\begin{proof} If $\pq\,{\in}\,\Z{\backslash}\{0\}$ and $X_{\pq}\,{\in}\gt_{\pq},$
then 
the inner derivation $\ad_{\gt}(X_{\pq})$ is nilpotent. Hence $\gt$
is decomposable if and only if the semisimple and nilpotent
parts of the $\ad_{\gt}(A),$ with $A\,{\in}\,\gt_{0},$ are still
inner derivations. This shows that $(i){\Leftrightarrow}(ii).$ \par  
Let us prove the equivalence $(ii){\Leftrightarrow}(iii).$ 
If $A\,{\in}\,\gt_{0},$ then $\ad_{\gt}(A)\,{\in}\,\Der_{0}(\gt)$ has a
Jordan-Chevalley decomposition $\ad_{\gt}(A)=D_{s}{+}D_{n}$ with 
$D_{s}$ semisimple and $D_{n}$ nilpotent, both belonging to 
$\K[\ad_{\gt}(A)].$
If  $A\,{=}\,A_{s}{+}A_{n}$ with $A_{s},A_{n}\,{\in}\,\gt_{0},$ 
then, by Lemma\ref{lm8.7},
we have 
$D_{s}\,{=}\,\ad_{\gt}(A_{s})$ and $D_{n}\,{=}\,\ad_{\gt}(A_{n})$ if and only if
$D_{s}|_{\mt}\,{=}\,\rhoup(A_{s})$ and $D_{n}|_{\mt}\,{=}\,\rhoup(A_{n}).$ 
This yields the last equivalence.
\end{proof}
Lemma~\ref{lemma8.13} tells us that for an almost effective $\gt$ 
\textit{being decomposable} is a property of its \textit{type} $\gt_{0}.$
\begin{prop} Let $\mt{=}{\sum}_{{\pq}<0}\gt_{\pq}$ be a finite dimensional
$\Z$-graded nilpotent Lie algebra and $\gt_{0}$ a 
Lie algebra of degree $0$
homogeneous derivations of~$\mt.$ 
Then the
following are equivalent:
\begin{itemize}
\item $\gt_0\,{\oplus}\,\mt$ is decomposable;
\item there is 
a decomposable prolongation of type $\gt_{0}$ of $\mt$;
\item all $\gt$ in $\Pp_{0}(\mt,\gt_{0})$ 
 are
decomposable.\qed
\end{itemize}
\end{prop}
\section{Reductive type} \label{sec-reduc}
Let $\gt{=}{\sum}_{\pq\in\Z}\,\gt_{\pq}$ be a finite dimensional $\Z$-graded Lie algebra,
$\nt$ its maximal nilpotent ideal and $\nt_{0}{\coloneqq}\nt\,{\cap}\,\gt_{0}.$ 
\begin{defn}
 We say that $\gt$ is of \emph{reductive type} if $\nt_{0}\,{=}\,\{0\}.$
\end{defn} 
By Theorem~\ref{thm8.12} we obtain 
\begin{prop} Let $\mt{=}{\sum}_{\pq<0}\,\gt_{\pq}$ be a finite dimensional $\Z$-graded
nilpotent Lie algebra, $\gt_{0}$ a Lie subalgebra of $\Der_{0}(\mt)$ and
$\Pp_{0}(\mt,\gt_{0})$ the collection of all finite dimensional
almost effective prolongations of type $\gt_{0}$ of $\mt.$ Then the following are equivalent: 
\begin{itemize}
 \item $\gt_{0}\,{\oplus}\,\mt$ is of reductive type;
 \item $\Pp_{0}(\mt,\gt_{0})$ contains a $\gt$ of reductive type;
 \item all $\gt$ in $\Pp_{0}(\mt,\gt_{0})$ are of reductive type.\qed
\end{itemize}
\end{prop}

 \begin{prop}
 Let $\gt$ be a finite dimensional 
 $\Z$-graded Lie algebra. If $\gt$ is of reductive type, 
 then $\gt_{0}$ is reductive and 
 the subalgebra $\rt_{0}$ of the degree $0$ 
homogeneous elements of its radical $\rt$ 
is contained in the center of $\gt_{0}.$ 
\end{prop} 
\begin{proof}
 Since $[\gt_{0},\rt_{0}]$ is contained in $\nt_{0},$ 
 the assumption that $\nt_{0}\,{=}\,\{0\}$
 implies that $\rt_{0}$ is contained in the center of  
 $\gt_{0}.$ By Theorem~\ref{lmc-dec}, $\gt$ has
 a $\Z$-graded semisimple Levi factor $\ses$
 and $\ses_{0}{\coloneqq}\ses\,{\cap}\,\gt_{0}$ 
 is reductive. 
 Then $\gt_{0}$ is reductive, being the sum of a
 reductive Lie subalgebra and of its centraliser. 
\end{proof}
 \par 
 In the  following lemma we do not require that $\gt$ is $\Z$-graded.
\begin{lem}\label{ll8.16}
 Let $\gt_{0}$ be a Lie subalgebra 
 of a finite dimensional Lie algebra $\gt$ 
 and $\sfW$ an $\ad_{\gt}({\gt_{0}})$-invariant subspace of $\gt.$ 
 If the representation of $\ad_{\gt}({\gt_{0}})$ on $\sfW$ 
 is semisimple, then 
 also the Lie subalgebra $\wt$ of $\gt$ generated by $\sfW$ is
a semismimple $\ad_{\gt}({\gt_{0}})$-module. 
If $\nt$ is the maximal nilpotent ideal of
$\gt,$ then $[\gt_{0}\,{\cap}\,\nt,\wt]\,{=}\,\{0\}.$ 
\end{lem}\begin{proof}
The tensor product representation of finite dimensional 
semisimple representations 
of a Lie algebra $\gt_{0}$ is semisimple (see e.g. \cite[Ch.I, \S{6}, Cor.1]{n1998lie}).
Therefore
each subspace 
$\ft_{\pq}(\sfW)$ 
of the free Lie algebra $\ft(\sfW)$ generated by
$\sfW$ is a semisimple $\gt_{0}$-module. 
Let $\varpi\,{:}\,\ft(\sfW)\,{\to}\,\wt$ be the natural projection. Since the image by
$\varpi$ of a finite dimensional 
irreducible $\gt_{0}$-submodule is either $\{0\}$ or semisimple,
it turns out that $\wt$ is a sum of irreducible $\gt_{0}$-submodules and hence
a semisimple $\gt_{0}$-module.
\end{proof}

\begin{prop}\label{pp8.13}
 Let $\gt$ be a  finite dimensional $\Z$-graded Lie algebra.
 If $\gt$ is almost effective,
 then 
the following are equivalent \begin{enumerate}
\item[($i$)]  $\gt$ is decomposable  
 and 
 of reductive type;
 \item[($ii$)] 
$\rhoup\,{:}\,\gt_0\,{\to}\,\gl_{\K}(\mt)$ is semisimple;
 \item[($iii$)] 
 $\gt$ is a semisimple $\ad_{\gt}(\gt_{0})$-module. 
\end{enumerate} 
\end{prop} 
\begin{proof} 
Clearly ($iii$)$\Rightarrow$($ii$) because
$\mt$ is $\ad_{\gt}(\gt_{0})$-invariant and $(i){\Rightarrow}(ii),(iii).$ \par
Let us show that $(ii){\Rightarrow}(i),(iii).$ \par
Let $\gt\,{=}\,\ses\,{\oplus}\,\rt$
be a Levi-Mal\v{c}ev decomposition of $\gt,$ with a $\Z$-graded semisimple
Levi factor~$\ses$. 
Since the adjoint representation restricts to a semi\-sim\-ple representation
of $\ses_{0}$ on $\gt,$ it suffices to show that the inner derivations corresponding
to the elements of $\rt_{0}$ are semisimple. Let indeed 
 $A\,{\in}\,\rt_{0}$ and
$D_{s},D_{n}\,{\in}\,\Der_{0}(\gt)$ be the semisimple and nilpotent summands of the Jordan-Chevalley
decomposition of $\ad_{\gt}(A).$ 
If we assume that $\rhoup$ is semisimple, 
by the uniqueness of the Jordan-Chevalley decomposition in $\gl_{\K}(\mt)$ we obtain
that $D_{s}{-}\ad_{\gt}(A)$ and $D_{n}$ vanish on $\mt.$ 
By Lemma~\ref{lm8.7} this yields $D_{n}\,{=}\,0$ and $\ad_{\gt}(A)\,{=}\,D_{s}$ is
a semisimple derivation on $\gt,$ proving 
at the same time that $(iii)$ holds, because $\ad_{\gt}(A)$ is semisimple for $A\,{\in}\rt_{0}$ and
that 
$(i)$ holds, because $\gt_{0}$ is the sum of a reductive ideal $\ses_{0}$ and an abelian
Lie algebra of semisimple elements.
\end{proof}
\begin{cor}
If $\mt$ is fundamental,  then $\gt$ 
 is decomposable 
 and 
 of the reductive type if and only 
 if $\rhoup_{-1}$ 
 is semisimple.\qed
\end{cor}

\section{Quasi-effective $\Z$-graded Lie algebras}
\label{sec-quasi}
Under the stronger assumption of quasi-effectiveness we obtain 
better structure theorems.
\begin{lem} \label{lm8.16}
Let $\gt$ be a quasi-effective finite dimensional $
\Z$-graded Lie algebra. If $\gt$ is
of reductive type, then the subspaces $\nt_{\pq}$ of  
 homogeneous elements of 
 positive degree $\pq$ of its maximal nilpotent ideal $\nt$ 
are trivial. 
\end{lem} 
\begin{proof}
Let us prove recursively that
$\nt_{\pq}\,{=}\,\{0\}$ for $\pq\,{\geq}\,0.$ 
For $\pq\,{=}\,0$ this is true by assumption.
Let $\pq\,{>}\,0$ and suppose 
we already know that $\nt_{\qq}\,{=}\,\{0\}$ for
$0{\leq}\qq{<}\pq.$ Then we have 
\begin{equation*}
 [\nt_{\pq},\mt]
 \subseteq{\sum}_{\qq<\pq}\nt_{\qq}
 \subseteq{\sum}_{\qq<0}\nt_{\qq}\subseteq\mt.
\end{equation*}
By the quasi-effectiveness assumption this yields 
$\nt_{\pq}\,{\subseteq}\,\nt_{0}\,{=}\,\{0\},$
proving the statement.
\end{proof}
We recall that the \emph{depth} of a $\Z$-graded Lie algebra 
$\gt$ is the largest
integer $\muup$ for which $\gt_{-\muup}\,{\neq}\{0\}.$ 
An almost effective
$\gt\,{\neq}\,\{0\}$ has positive depth.
\begin{prop}\label{pp8.18}
Let $\gt$ be a quasi-effective 
finite dimensional $\Z$-graded Lie algebra.
If $\gt$ is 
of reductive type
and depth $\muup,$ then $\gt_{\pq}\,{=}\,\{0\}$
 for $\pq\,{>}\,\muup.$ 
\end{prop} 
\begin{proof} We note that $\gt_{\pq}$ is orthogonal to 
$\gt,$ and hence belongs to
$\rt_{\pq},$ for $\pq{>}\muup.$ For positive $\pq$ we have 
$\rt_{\pq}\,{=}\,\nt_{\pq}$
and thus the statement is a consequence of 
Lemma~\ref{lm8.16}. 
\end{proof}
A finite dimensional $\Z$-graded Lie algebra $\gt$ has a $\Z$-graded solvable radical $\rt,$
whose depth $\muup_{r}$ we call its \emph{solvable depth}; besides, all
its $\Z$-graded semisimple Levi factors have the 
same depth $\muup_{s},$ that we
call its \emph{semisimple depth}. In case $\gt$ is solvable, we set $\muup_{s}{=}0$
and, likewise, we set $\muup_{r}{=}0$ when
it is semisimple. 
We have the following  
\begin{thm} \label{thm8.22}
A finite dimensional $\Z$-graded quasi-effective
Lie algebra of reductive type 
has a $\Z$-graded Levi-Mal\v{c}ev decomposition 
\begin{equation}\label{q8.17}
\gt=\ses\oplus\,\rt
\end{equation}
with a $\Z$-graded semisimple Levi factor 
\begin{equation}\label{q8.18}
\ses={\sum}_{\pq=-\muup_{s}}^{\muup_{s}}\ses_{\pq}
\end{equation}
and a $\Z$-graded  radical  
\begin{equation}\label{q8.19}
\rt={\sum}_{\pq=-\muup_{r}}^0\rt_{\pq}
\end{equation}
with no homogeneous term of positive degree.
\par 
Moreover, $\gt$ is decomposable if and only if 
$\rhoup(A)$ 
is semisimple on $\mt$ 
for all $A\,{\in}\,\rt_0.$ When $\gt$ is decomposable, quasi-effective  
and of the reductive type, 
the $\Z$-graded Levi factor $\ses$ can be
chosen in such a way that
\begin{equation}\label{q8.20}
[\rt_0,\ses]=\{0\}.
\end{equation}
\end{thm}
\begin{proof} The fact that
for every $\Z$-graded Levi-Mal\v{c}ev
decomposition \eqref{q8.17} the gradings of the Levi factor
$\ses$ and of the solvable radical 
$\rt$ are as in \eqref{q8.18} and
\eqref{q8.19} follows from Lemma~\ref{lm8.16}. 
By Proposition~\ref{pp8.13} (see also \cite[Thm.2, Ch.VII, \S{5}]{Bou82})
a quasi-effective and therefore an  almost effective $\gt$ is decomposable if and only
if $\rhoup(\rt_{0})\,{\subseteq}\,\gl_{\K}(\mt)$ is decomposable. 
\par 
If the elements of $\rt_{0}$ are $\ad$-semisimple, then for each $A\,{\in}\,\rt_{0}$
we have $\gt\,{=}\,[A,\gt]{\oplus}\{X\,{\in}\,\gt\,{\mid}\,[A,X]\,{=}0\}.$
Hence $\kt\,{=}\,\{X\,{\in}\,\gt\,{\mid}\,[X,\rt_{0}]\,{=}\,\{0\}\}$ is a $\Z$-graded
Lie subalgebra of $\gt$ such that $\gt\,{=}\,\kt{+}\rt$ and
a $\Z$-graded semisimple Levi factor of $\kt$ is then a $\Z$-graded semisimple
Levi factor of $\gt$ satisfying \eqref{q8.20}. 
This completes the proof.
\end{proof}
\begin{cor}
 Assume that $\gt$ is an effective 
 finite dimensional $\Z$-graded Lie algebra
 with $\mt$ fundamental.
 If $\rhoup_{-1}$ is simple, 
 either $\gt$ is semisimple
 or $\nt\,{=}\,\mt$ and  $\gt\,{=}\,\mt\,{\oplus}\gt_{0}.$ 
\end{cor} 
\begin{proof} By Proposition~\ref{pp8.13} and 
Lemma~\ref{ll8.16} we know that $\gt$ is decomposable
and of reductive type. Consider a $\Z$-graded Levi-Mal\v{c}ev
decomposition \eqref{q8.17} for which 
\eqref{q8.18},
\eqref{q8.19}, \eqref{q8.20} are satisfied. Since we assumed that
$\rhoup_{-1}$ is irreducible, 
either $\mt\,{\subseteq}\,\ses,$
or $\mt\,{\subseteq}\,\rt.$  
In the first case $\rt\,{=}\,\rt_{0}\,{=}
\{0\}$ by the assumption that $\gt$ is effective, 
because $[\rt_{0},\mt]=0$ and thus $\gt\,{=}\,\ses.$ 
In the second case we have $\ses\,{\subseteq}\,\gt_{0}$
and  $\mt\,{=}\,\mt\,{\oplus}\,\gt_{0},$ because
$\gt_{\pq}\,{=}\,\nt_{\pq}\,{=}\,\{0\}$ for $\pq\,{>}\,0.$ 
\end{proof}
\begin{exam} For a positive integer $\hq{>}1,$ we consider on 
$\sfV\,{=}\,\K^{3}$ the $\Z$-gradation described by
$E_{\sfV}{=}\!\left( 
\begin{smallmatrix}
 0\\ &\hq \\ && 1
\end{smallmatrix}\right).$
Then the upper triangular nilpotent $3{\times}3$ matrices are a
$\Z$-graded Lie subalgebra 
$\gt{=}{\sum}_{\pq\in\Z}\,\gt_{\pq}$ with $\gt_{\pq}{=}\{0\}$ for
$\pq\neq{-}\hq, \hq{-}1, {-}1$ and 
\begin{equation*}
 \gt_{-\hq}=\left\{\left( 
\begin{smallmatrix}
 0 & \lambdaup & 0\\
 0 & 0 & 0\\
 0 & 0 & 0
\end{smallmatrix}\right)\right\},
\;\;
 \gt_{-1}=\left\{\left( 
\begin{smallmatrix}
 0 & 0 & \lambdaup\\
 0 & 0 & 0\\
 0 & 0 & 0
\end{smallmatrix}\right)\right\},\;\;
 \gt_{\hq-1}=\left\{\left( 
\begin{smallmatrix}
 0 & 0 & 0\\
 0 & 0 &\lambdaup\\
 0 & 0 & 0
\end{smallmatrix}\right)\right\},
\end{equation*}
 which is almost, but not quasi-effective. Moreover, $\gt$  is trivially of reductive type,
because $\gt_{0}{=}\{0\}.$ Note that 
 \eqref{q8.19} is not valid in this case. 
\end{exam}
\begin{exam}
 Let $\sfV\,{=}\,\K^{4}$ and, for a pair $\hq,\kq$ of positive integers,  with $\hq{+}\kq{>}2,$ 
 consider on $\sfV$ the $\Z$-gradation provided by
 $E_{\sfV}=\left( 
\begin{smallmatrix}
 0 \\ & \hq \\ && {-}\kq \\ &&& 1
\end{smallmatrix}
 \right).$ We consider the $\Z$-graded subalgebra $\gt$ 
 of $\glc_{\K}(\sfV)$ consisting of the
 upper triangular nilpotent $4{\times}4$ matrices. Then $\gt{=}{\sum}_{\pq\in\Z}\,\gt_{\pq}$
 with \par\centerline{$\gt_{\pq}\,{=}\,\{0\}$ for 
 $\pq\notin\{ {-}\hq, -(1{+}k),  {-}1, \kq, (\hq{-}1), (\hq{+}\kq)\}$ and} 
\begin{align*}
 \gt_{-h}=\left\{\left( 
\begin{smallmatrix}
 0 & \lambdaup & 0 & 0\\
 0 & 0 & 0 & 0 \\
  0 & 0 & 0 & 0 \\
   0 & 0 & 0 & 0 
\end{smallmatrix}\right)\right\},\;
 \gt_{-\kq-1}=\left\{\left( 
\begin{smallmatrix}
 0 & 0 & 0 & 0\\
 0 & 0 & 0 & 0 \\
  0 & 0 & 0 & \lambdaup \\
   0 & 0 & 0 & 0 
\end{smallmatrix}\right)\right\},\;
 \gt_{-1}=\left\{\left( 
\begin{smallmatrix}
 0 & 0 & 0 & \lambdaup\\
 0 & 0 & 0 & 0 \\
  0 & 0 & 0 & 0 \\
   0 & 0 & 0 & 0 
\end{smallmatrix}\right)\right\},\\
 \gt_{\kq}=\left\{\left( 
\begin{smallmatrix}
 0 & 0 & \lambdaup & 0\\
 0 & 0 & 0 & 0 \\
  0 & 0 & 0 & 0 \\
   0 & 0 & 0 & 0 
\end{smallmatrix}\right)\right\},\;
  \gt_{\hq-1}=\left\{\left( 
\begin{smallmatrix}
 0 & 0 & 0 & 0\\
 0 & 0 & 0 & \lambdaup \\
  0 & 0 & 0 & 0 \\
   0 & 0 & 0 & 0 
\end{smallmatrix}\right)\right\},\;
 \gt_{\hq+\kq}=\left\{\left( 
\begin{smallmatrix}
 0 & 0 & 0 & 0\\
 0 & 0 & \lambdaup & 0 \\
  0 & 0 & 0 & 0 \\
   0 & 0 & 0 & 0 
\end{smallmatrix}\right)\right\}.
\end{align*}
Then $\gt$ is almost, but not quasi-effective. Moreover, $\gt$ 
is of reductive type if and only if $\hq{>}1.$
Its depth is $\sup\{\hq,\kq{+}1\},$ which, if $\kq,\hq{>}1,$  is strictly smaller than  
the maximum index $\pq$ for which $\gt_{\pq}{\neq}\{0\},$ which is $\hq{+}\kq.$ 
Also in this case \eqref{q8.19} is not valid.
\end{exam}
\begin{exam}
We will be primarily interested in finite dimensional 
quasi-effective $Z$-graded 
$\gt$ having a \textit{reductive} {structure algebra}~$\gt_{0}.$ 
Before getting into this, let us briefly discuss 
how to construct general 
finite dimensional nilpotent prolongation
. By Engel's theorem, 
a faithful representation
of $\gt$ will lead to realise $\gt$ as a Lie algebra  
of nilpotent upper triangular
$n{\times}n$ matrices 
\begin{equation*} 
\begin{pmatrix}
 0 & x_{1,2}& x_{1,3} & \hdots & x_{1,n-1}&x_{1,n}\\
 0 & 0 & x_{2,3}& \hdots & x_{2,n-1}& x_{2,n}\\
 \vdots & \ddots & \ddots & \ddots & \ddots &\vdots \\
 \vdots & \ddots & \ddots & \ddots & \ddots& \vdots \\
 0 & 0 & 0 &\hdots & 0 & x_{n-1,n}\\
 0 & 0 & 0 & \hdots & 0 & 0 
\end{pmatrix}.
\end{equation*}
\par 
A grading on $\gt$ can be described by the action of an integral  
diagonal matrix ${E}{=}\diag(k_{1},\hdots,k_{n})$:
the  
entry $x_{i,j}$ is then \textit{homogeneous} of 
degree~$k_{i}{-}k_{j}.$ \par 
Let us consider the specific example consisting of the 
Lie algebra $\gt$ of 
$8{\times}8$ upper triangular matrices
with entries in $\K$ which are antisymmetric with respect 
to the second diagonal. 
These  matrices belong to  the orthogonal algebra
of the bilinear symmetric form $\bil(\vq,\wq)
\,{=}\,\vq^{\intercal}\sfJ\wq,$ 
with  
\begin{equation*}
 \sfJ= \left( \begin{smallmatrix}
 &&&&&&& 1\\
 &&&&&& 1 \\
 &&&&& 1\\
 &&&& 1\\
 &&& 1 \\
 && 1 \\
 & 1\\
 1
\end{smallmatrix}\right)
\end{equation*}
and are indeed the largest nilpotent ideal of 
a Borel subalgebra of a split form of $\mathbf{D}_{4}.$ 
By using 
the gradation with ${E}=\diag(-1,-2,-1,0,0,1,2,1),$ 
we obtain on $\gt$ a structure of  $12$-dimensional nilpotent 
fundamental graded Lie algebra  
\begin{gather*} \gt\,{=}\,{\sum}_{\pq=-3}^{1}\gt_{\pq},
\;\;\;\text{with}\\
\dim_{\K}(\gt_{-3}){=}2, \; \dim_{\K}(\gt_{-2}){=}3,\;
 \dim_{\K}(\gt_{-1}){=}5,\; \dim_{\K}(\gt_{0}){=}1,\;
\dim_{\K}(\gt_{1}){=}1.
\end{gather*}
\par
The nonzero elements of $\gt_{0}$ have rank $2$ 
on $\gt_{-1}$ and therefore the maximal prolongation of type
$\gt_{0}$ of $\mt{=}{\sum}_{\pq=-3}^{-1}\gt_{\pq}$ 
is finite dimensional
(see e.g. \cite{MMN2018}). 
We obtain a solvable prolongation of $\gt$ by adding to $\gt_{0}$ 
the $8{\times}8$ diagonal matrices which are
antisymmetric with respect to the second diagonal. 
\end{exam}
\section{Semisimple prolongations}\label{sect9}
By Theorem~\ref{thm8.22} all positive degree summands of an effective
$\Z$-graded finite dimensional Lie algebra $\gt$ 
of reductive type 
are contained in its graded semisimple Levi factor. 
It is therefore of some interest 
investigating the way 
$\Z$-gradations of semisimple Lie algebras 
relate to (maximal) prolongations of fundamental
graded Lie algebras. 
\par\smallskip
Let $\gt={\sum}_{\pq={-}\muup}^\muup\gt_{\pq}$ 
be a finite dimensional $\Z$-graded semisimple Lie algebra over $\K.$ 
We say that its gradation is \textit{not trivial} if $\muup{>}0$ and
$\gt_{\muup}{\neq}\{0\}.$ Moreover,  $\gt$ is effective iff none of 
its nontrivial ideals is contained in
$\gt_0.$  
Its Lie subalgebra $\gt_0$ 
is reductive, since the restriction 
 of the Killing form of $\gt$ to $\gt_{0}$
 is nondegenerate.
Set $ \mt={\sum}_{\pq={-}\muup}^{{-}1}\gt_{\pq}$ and 
$V{=}\gt_{{-}1}.$ 
Assuming 
that $\mt$ is fundamental,     
$\gt$ is an effective prolongation 
of $\mt$ if and only if the action of
 $\gt_0$ on $V$ is faithful: in this case $\gt_0$ 
 can be identified with a Lie subalgebra of $\gl_{\K}(V).$ 
The derived algebra $
[\gt_0,\gt_0]$ 
is semisimple and $V$ decomposes into a direct sum
 $V=V_1{\oplus}\cdots{\oplus}V_k$ 
 of its irreducible representations.  
 For each $1{\leq}i{\leq}k,$ 
 the set $\Aa_i$ of $\K$-endomorphisms of $V_i$ which commute with 
the action of $[\gt_0,\gt_0]$ is, by Schur's lemma, 
a division $\K$-algebra. The elements of $\Aa_i$ 
uniquely extend to derivations 
 of $\gt$ vanishing on $V_{\!{j}}$ for $j{\neq}i.$ 
 In particular, the identity of $\Aa_i$ yields a projection
$\etaup_{\,i}:V{\to}V_i.$     
 Since every derivation 
 of a semisimple finite dimensional Lie algebra is inner, 
we can consider the $\Aa_i$'s as  subalgebras of~$\gt_0.$

\begin{lem} Assume that $\gt$ is finite dimensional and
semisimple and that 
the action of $\gt_0$ on $V$ is faithful. Then the center of 
 $\Aa_1\oplus\cdots\oplus\Aa_k$ is the center of $\gt_0.$
\qed
\end{lem} 
\begin{rmk} The commutant  
$\Aa_i$ 
may contain a simple Lie algebra over~$\K,$ that will 
contribute as a summand to $[\gt_0,\gt_0].$ 
For instance, when $\K$ is the field $\R$ of real numbers,
the possible $\Aa_i$ are $\R$ itself, 
or $\C,$ or the non commutative real division algebra 
$\Hb$ of quaternions,
and $\Hb{\simeq}\ot(3){\oplus}\R.$  \par
A simple instance of this situation is the simple Lie algebra
$\slt_2(\Hb),$ with $\gt_0{=}
\ot(3){\oplus}\ot(3){\oplus}\R{\simeq}\ot(4){\oplus}\R{\simeq}
\mathfrak{co}(4),$ with the standard action 
on $V{=}\R^4{\simeq}\Hb,$ where $\slt_2(\Hb)$ can be viewed
as a maximal prolongation of type $\mathfrak{co}(4)$ of~$V.$  
This presentation corresponds to the 
\textit{cross marked Satake
diagram} (see e.g.~\cite{AMN06})
\begin{equation*}\vspace{-15pt}
  \xymatrix@R=-.3pc{ 
\!\!\medbullet\!\!\!
\ar@{-}[r]&\!\!\medcirc\!\! \ar@{-}[r]&   \!\!\medbullet\\
&{\,\times}}
\end{equation*}
\par\medskip
\end{rmk}
\par\medskip
Semisimple and maximal prolongations  are 
related by the following proposition (see e.g. \cite{MN97,Tan67}).
\begin{prop} \label{prop-7.3}
Assume that $\gt$ is semisimple and  that 
the action of $\gt_0$ on $V$ is faithful.
Then $\gt$ is maximal
among the finite dimensional effective prolongations of type 
$\gl_{\K}(V)$ of $\mt.$ 
\end{prop} 
\begin{proof}
 Indeed, an effective
  finite dimensional  
  prolongation $\G{=}{\sum}_{\pq{\in}\Z}\G_{\pq}$ 
 of $\mt$ containing $\gt$ is a $\gt$-module. 
 Since the finite dimensional linear representations of
 a semisimple $\gt$ are completely reducible,
 $\gt$ has in $\G$ a complementary $\Z$-graded $\gt$-module~$\G'.$  
 The conditions that $[\G',\mt]{\subset}\G'$ and that $\G'$ 
 is contained in $\G_{+}{=}{\sum}_{\pq{\geq}0}\G_{\pq}$ 
 implies by effectiveness that $\G'=\{0\}.$  
\end{proof}
\begin{exam}
If $\mt\,{=}\,\gt_{-1}\,{=}\,\K^{n},$ then $\slt_{n+1}(\K)$ is the 
unique finite dimensional  
semisimple prolongation  
of $\mt.$ However, the maximal prolongation of type
$\gl_{n}(\K)$ of $\mt$ is the infinite dimensional 
graded Lie algebra $\Xx(\K^{n})$
of vector fields with polynomial coefficients in $\K^{n}$
(see e.g. \cite[\S{3}]{MMN2018}).
\end{exam}
In the rest of this section, we will exhibit
structures of maximal $\Z$-graded prolongations
on semi\-sim\-ple Lie algebras,
that we think could be 
of some interest
in geometry and physics. \par
Let us explain the pattern of our constructions.
We start from a semismiple Lie algebra $\Li_0$
and fix a
faithful finite dimensional $\Li_0$-module $V,$ 
identifying $\Li_0$
with a Lie subalgebra of $\gl_{\K}(V).$ 
The structure algebra $\gt_0$ is obtained by adding to
$\Li_0$ its commutant in $\gl_{\K}(V).$ 
Then $\gt_0$ is 
reductive, 
with $\Li_0\,{\subseteq}\,[\gt_0,\gt_0].$ 
The derived algebra $\Li{\coloneqq}[\gt_0,\gt_0]$ is the
semisimple ideal of $\gt_0.$  
The exterior power $\Lambda^2(V)$
is  an $\Li$-module and  we 
can choose $\gt_{{-}2}$ equal to 
any $\Li$-submodule of $\Lambda^2(V).$ 
Likewise, all homogeneous summands in 
 the natural gradation of the free Lie algebra
$\ft(V){=}{\sum}_{\pq<0}\ft_{\pq}(V)$ of $V$ 
(cf. \cite{MMN2018,Reu93,Warhurst2007})
are $\Li$-modules
and we can choose 
$\gt_{{-}3}$ as an $\Li$-submodule of 
$(V{\otimes}\gt_{{-}2})\cap\ft_{{-}3}(V),$ and next, 
by recurrence, $\gt_{{-}\pq{-}1}$ as
an $\Li$-submodule of 
the intersection
$(V{\otimes}\gt_{{-}\pq})\cap\ft_{{-}\pq{-}1}(V).$ 
In this way we build up the general fundamental graded Lie algebra
$\mt$ 
with $\gt_{{-}1}{=}V$ admitting a maximal prolongation 
with 
structure subalgebra $\Li.$
\par
A randomly constructed $\mt$ may have, in general, 
no $\Z$-graded prolongation of positive heigth.
For this, we need 
to make specific choices of~$\mt.$ 
\par 
Let $\K$ be the field $\C$ of complex numbers. Fix a
Cartan subalgebra $\hg$ 
of $\Li$ and let $\Rad,$ $\Wi$ be the corresponding root
system and weight lattice, which  are contained 
in a Euclidean space $\R^{\ell}.$ 
For an $\omegaup\,{\in}\,\Wi,$ we denote by $\Wi(\omegaup)$
the set of weights of an irreducible $\Li$-module with
extremal weigth $\omegaup.$ \par 
As usual, we denote by $(\alphaup|\betaup)$ 
the standard scalar product of 
$\alphaup,\betaup\in\R^\ell$ and, for a nonzero element 
$\alphaup$ of $\R^\ell$ we set
$\alphaup^\vee=2\alphaup{/}\|\alphaup\|^2$ 
and 
$\langle\alphaup|\betaup\rangle=(\alphaup|\betaup^\vee).$ 
Then,
by choosing a system $\alphaup_1,\hdots,\alphaup_{\ell}$ 
of simple roots in $\Rad,$
we associate to $\Li$ the Cartan matrix $A{=}
(\langle\alphaup_i\,|\,\alphaup_j\rangle)_{1{\leq}i,j{\leq}\ell}.$ 
\par 
Our next step is to
construct, starting from  the data of 
$\Li$ and $V$ a \textit{generalised Cartan matrix} 
$\tilde{A}$ extending $A$
(see \cite{kac_1990, Moody68}). \par
The matrix $\tilde{A}$ will be associated to an
isometric embedding of $\R^{\ell}$ 
into an orthogonal 
space $(\R^{\ell{+}k},\bil)$ ($k$ is the number
of irreducible $\Li$-modules $V_i$ in $V$). 
Here $\bil$ is a symmetric bilinear form on $\R^{\ell{+}k},$
which restricts to the Euclidean scalar product on $\R^\ell.$
For each 
$V_i,$ 
let $\omegaup_i$ be its \textit{lowest} weight (i.e. 
the one with $(\omegaup_i|\alphaup_j){\leq}0$
for all $j{=}1,\hdots,\ell$) 
and define a new \textit{simple root}
$\alphaup_{\ell{+}i}=\omegaup_i{+}\epi_i$ by adding to $\omegaup_i$ a \textit{marker} 
$\epi_i$ from $\mathrm{E}{=}\{\wq{\in}\R^{\ell{+}k}
\mid \bil(\vq,\wq){=}0,\;\forall\vq{\in}\R^\ell\}.$ 
We choose linearly independent $\epi_1,\hdots,\epi_k$ in $\mathrm{E}$ such that
$\bil(\alphaup_i,\alphaup_i){\neq}0$
for all $1{\leq}i{\leq}\ell{+}k$ and set
\begin{equation*}
 \tilde{A}=
 (a_{i,j})_{1{\leq}i,j{\leq}\ell{+}k}
 =\left(\dfrac{2{\bil}(\alphaup_i,\alphaup_j)}{{\bil}(\alphaup_i,\alphaup_i)}\right)_{1{\leq}i,j{\leq}\ell{+}k}.
\end{equation*}
Then $A$ is the submatrix 
of the first $\ell$ lines and columns of $\tilde{A}.$ 
The elements $a_{i,j}$ with $i{\leq}\ell$ are 
determined
by $\Li$ and
$V,$  because $a_{i,j}{=}\langle\alphaup_i\,|\,\omegaup_{j{-}\ell}\rangle$
if $\ell{<}j{\leq}\ell{+}k.$ The others
depend on the choice of $\bil,$ 
for which we need to keep the constrain that
$\tilde{A}$ be a generalised Cartan matrix,   
having  nonpositive integers off the main diagonal.
The restriction of 
$\bil$  to $\mathrm{E}$ must satisfy
\begin{equation} \label{eq7.1}
\begin{cases}
\|\alphaup_{\ell{+}i}\|^2{=} \|\omegaup_i\|^2{+}\bil(\epi_i,\epi_i)>0,&\text{for $1{\leq}i{\leq}k,$}
\\[5pt]
\dfrac{2(\omegaup_i\,|\,\alphaup_j)}{\|\omegaup_i\|^2{+}\bil(\epi_i,\epi_i)}\in\Z, & \text{for
$1{\leq}i{\leq}k,$ $1{\leq}j{\leq}\ell,$}\\[14pt]
\dfrac{2(\omegaup_i\,|\,\omegaup_j){+}2\bil(\epi_i,\epi_j)}{\|\omegaup_i\|^2{+}\bil(\epi_i,\epi_i)}\in\Z,
&\text{for $1{\leq}i{\neq}j{\leq}k,$}\\[12pt]
{(\omegaup_i\,|\,\omegaup_j){+}\bil(\epi_i,\epi_j)}{\leq}0,
&\text{for $1{\leq}i{\neq}j{\leq}k.$}
 \end{cases}
\end{equation}
Then $\gt_0{=}\Li,$ $\gt_{{-}1}{=}V,$ for a $\Z$-gradation
of the
Kac-Moody algebra $\gt$ 
of $\tilde{A},$ which
is finite dimensional  iff $\bil$
is positive definite (see e.g. \cite{kac_1990}).  
In general, we may consider choices for which $\bil$ has maximal
positive/non-negative inertia.
Note that $\bil$ is completely determined by the values
of the entries of
\eqref{eq7.1}. When $\gt$ is infinite dimensional, the maximal 
effective prolongation 
of type $\gt_0$ of
$\mt$ is in general strictly larger than $\gt.$
\begin{rmk}[Simple prolongations] \label{rmk7.4}
When $V$ is irreducible,
with extremal weight $\omegaup,$ we need to add a unique marker 
$\epi.$ Then
conditions \eqref{eq7.1} reduce to the first two and
$\tilde{A}$ is completely determined 
by the value of $\bil(\epi,\epi).$ 
We require that 
\begin{equation}\label{eq7.2}
 \|\omegaup\|^2{+}\bil(\epi,\epi){>}0
 \;\;\text{and}\;\; \dfrac{2(\alphaup|\omegaup)}{ \|\omegaup\|^2
 {+}\bil(\epi,\epi)}
 \in\Z,\;\;\forall\alphaup{\in}\Rad.
\end{equation}
In particular, to obtain a finite dimensional semisimple 
$\gt,$ 
we need that 
$\|\omegaup\|^2<\min\{2|(\omegaup|\alphaup)|>0
\mid\alphaup{\in}\Rad\}.$ 
\end{rmk}
To discuss the case where $\K{=}\R,$
we observe that the complexification of the 
effective prolongation 
of a real fundamental
graded Lie algebra $\mt$ is the complex 
effective prolongation 
of the complexification of $\mt.$ 
In this way we reduce to the complex case,
by taking into account the way real representations
lift to complex ones (see e.g. \cite[Ch.IX,Appendix]{Bou82}).
\begin{rmk}  
Semisimple effective prolongations 
$\gt$ can be read off 
their associated diagrams of Dynkin/Satake.  
In particular, 
for a Satake diagram 
$\Sigma$ (see e.g. \cite{AMN06, Ara62})
of a  real semisimple effective prolongation 
$\gt$  
we require that
\begin{itemize}
 \item[(i)] the eigenspaces corresponding to  fundamental roots of $\Sigma$ are homogeneous of degree
 either $0$ or $1;$
\item[(ii)] compact roots have degree $0$ and those  joined by an
arrow have the same degree;  
\item[(iii)] degree $0$ roots 
are the nodes of the Satake diagram of $[\gt_0,\gt_0].$ 
\end{itemize}
Crosses can be added under the nodes of a Satake diagram to indicate the roots of positive degree.
Complex type representations can be associated to couples of \textit{positive} roots joined by
an arrow (see e.g.  \cite{MN98}).\end{rmk}
\par  
\smallskip

Exceptional Lie algebras 
naturally arise as maximal effective prolongations 
of fundamental graded Lie algebras with non exceptional
structure algebras. 
These constructions are related to the investigation of 
 their maximal rank reductive subalgebras
(see e.g. \cite{adams1996lectures,golubitsky1971}). 

\subsection{Structure algebras of type $\mathbf{A}$} To describe the root system of
$\slt_n(\C),$ it is convenient to use
an orthonormal basis $\eq_1,\hdots,\eq_n$ of $\R^n$ and set  
\begin{equation*}
 \Rad=\{{\pm}(\eq_i{-}\eq_j)\mid 1{\leq}i{<}j{\leq}n\}.
\end{equation*}
The Dynkin diagram is
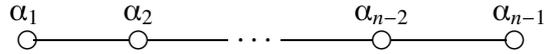
\begin{equation*}
  \xymatrix@R=-.3pc{  \alphaup_1 & \alphaup_2 &  &\alphaup_{n{-}2} &\alphaup_{n{-}1} \\
\!\!\medcirc\!\!\!
\ar@{-}[r]&\!\!\medcirc\!\! \ar@{-}[r]& \cdots \ar@{-}[r]& \!\!\medcirc\!\!\ar@{-}[r]  &\!\!\medcirc}
\end{equation*}
with simple roots  
\begin{equation*}
 \alphaup_i=\eq_{i}{-}\eq_{i{+}1},\;\;\text{for}\;\; 1{\leq}i{\leq}n{-}1.
\end{equation*}
Set  $\eq_0{=}\eq_1{+}\cdots{+}\eq_n.$ 
The simple positive weights 
in $\langle\eq_i{-}\eq_j\,
{\mid} 1{\leq}i{<}j{\leq}n\rangle
{\simeq}\R^{n{-}1}$ 
are 
$\omegaup_{k}{=}{\sum}_{j{=}1}^{k}\eq_{i}{-}\frac{k}{n}\eq_0,$ 
for $1{\leq}i{\leq}n{-}1,$
with $\omegaup_k$ corresponding to the irreducible representation 
$\Lambda^k(\C^n).$ 
We have 
\begin{align*} (\omegaup_j|\omegaup_k)
=j{\cdot}\left(1-\frac{k}{n}\right)>0,\;\;\forall 1{\leq}j{\leq}k{<}n
\end{align*}
and hence, for a dominant weight ($\aq_j{\geq}0$ for all $j$),
\begin{align*}
 \left\| {\sum}_{j{=}1}^{n{-}1}\aq_j\omegaup_j\right\|^2
 ={\sum}_{j{=}1}j\aq_j\left[\aq_j\left(1{-}\frac{j}{n}\right)
 {+}2{\sum}_{k{=}j{+}1}^{n{-}1}\aq_k
 \left(1{-}\frac{k}{n}\right)\right]\\
 > {\sum}_{j{=}1}^{n{-}1}\aq_j^2{\cdot} j{\cdot}
 \left(1-\frac{j}{n}\right).
\end{align*}
By Remark~\ref{rmk7.4}, to find a semisimple
effective fundamental graded Lie algebra 
 $\gt$ with $\gt_0{=}\gl_n(\C)$ and 
$\gt_{{-}1}$ equal to 
an irreducible $\slt_n(\C)$-module $V{=}V_{\omegaup},$ we need that 
$\omegaup{\sim}\omegaup_i$ with 
\begin{itemize}
\item either $i\,{=}\,1, n{-}1,$ and any $n{\geq}{2},$ 
\item or
 $i{=}2,n{-}2$ and $n{\geq}{4},$ 
\item
or 
 $i{=}3,n{-}3$ and $6{\leq}n{\leq}9.$  
\end{itemize}
A \textit{marker}  
could be taken of the form $\epi{=}\cq{\cdot}\eq_{\,0}.$ 

\subsubsection{$\gt_{{-}1}{\simeq}V_{\omegaup_1}:$ construction of 
$\mathbf{B}_n$ and $\mathbf{G}_2$}
\label{sub7.1.1}
Let  $\omegaup{=}\eq_n{-}\tfrac{1}{n}\eq_0{\sim}\omegaup_1.$ 
By Remark~\ref{rmk7.4}, 
the possible choices for
$\epi$ are 
\begin{equation*}
 \epi = 
\begin{cases}
\tfrac{\sqrt{1{+}n}}{n}\eq_0 \Longrightarrow \|\omegaup{+}\epi\|^2
{=}2,\\
 \tfrac{1}{n}\eq_0 \Longrightarrow \|\omegaup{+}\epi\|^2{=}1,\\
 \tfrac{1}{2\sqrt{3}}(\eq_1{+}\eq_2) \Longrightarrow \|\omegaup
 {+}\epi\|^2{=}\tfrac{2}{3}, & \text{if $n{=}2.$}
\end{cases}
\end{equation*}
The first corresponds to an $\mt$ of kind $1,$ the second to an 
$\mt$ of kind $2,$ the third
to an $\mt$ of kind $3.$\par 
The kind one abelian Lie algebra $\mt{=}V_{\omegaup_1}{\simeq}\C^n$ 
has the simple effective prolongation 
$\slt_{n{+}1}(\C),$ but has 
the infinite dimensional maximal effective prolongation
$\Xx(V)$ (see e.g. \cite[\S{3}]{MMN2018}).\par
Since
$\Lambda^2(V_{\omegaup_1})$
is an
irreducible $\slt_n(\C)$-module, with dominant weight $\omegaup_2,$ 
the only possible choice for an $\mt$ of kind two is
to take $\mt\,{=}\,\C^n{\oplus}\Lambda^2(\C^n).$
Indeed, with $\epi=\tfrac{1}{n}\eq_0,$ we have 
\begin{equation*}
 \|\,\omegaup_1{+}\epi\,\|^2{=}\|\,\eq_1\,\|^2{=}1,
 \quad \|\,\omegaup_2{+}2\epi\,\|^2=\|\eq_1{+}\eq_2\|^2=2.
\end{equation*}
 By \cite[Cor.4.11]{MMN2018}
when $n{\leq}2,$ $\dim(\gt_{{-}2}){\leq}1$ and thus the
maximal effective prolongation
of this $\mt$ is infinite dimensional.
It is finite dimensional by 
\cite[Thm.4.8]{MMN2018}
when
$n{\geq}3,$  
and in fact is
isomorphic to $\ot(2n{+}1,\C)$ (see e.g.~\cite{Yam93}).
Indeed, setting 
\begin{equation*} 
\begin{cases}
 \Rad_{\;{-}2}=\{\omegaup+2\epi\mid \omegaup{\in}\pes(\omegaup_2)\}=\{\eq_i{+}\eq_j\mid 1{\leq}i{<}j{\leq}n\},\\
 \Rad_{\;{-}1}=\{\omegaup+\epi\mid 
 \omegaup{\in}\pes(\omegaup_1)\}=\{\eq_i\mid 1{\leq}i{\leq}n\},\\ 
 \Rad_{\;0}= \{\eq_i{-}\eq_j\mid 1{\leq}i{\neq}j{\leq}n\},\\
 \Rad_{\;{}1}=\{\omegaup-\epi\mid \omegaup{\in}
 \pes(\omegaup_{n{-}1})\}=\{{-}\eq_i\mid 1{\leq}i{\leq}n\},\\  
  \Rad_{\;2}=\{\omegaup-2\epi\mid \omegaup
  {\in}\pes(\omegaup_{n{-}2})\}=\{{-}\eq_i{-}\eq_j
  \mid 1{\leq}i{<}j{\leq}n\},
\end{cases}
\end{equation*}
 the union $\Rad={\bigcup}_{\pq{=}{-}2}^2\Rad_{\;\pq}$ is a root system of type $\mathbf{B}_n$ 
 and the prolongation 
\begin{equation*}
 \gt={\sum}_{\pq{=}{-}2}^2\gt_{\pq}\simeq\ot(2n{+}1,\C),\;\;\text{with}\;\; \gt_0
 {=}\hg_n{\oplus}\langle\Rad_{\;0}\rangle,\;\;
 \gt_{\pq}=\langle\Rad_{\;\pq}\rangle,
\end{equation*}
where $\hg_n$ is an $n$-dimensional Cartan subalgebra, is by Proposition~\ref{prop-7.3},
a maximal effective prolongation 
of $\mt,$ because
$\ot(2n{+}1,\C)$ is simple. 
We  got indeed
\begin{equation*}
 \Rad=\{{\pm}\eq_i\mid 1{\leq}i{\leq}n\}
 \cup\{{\pm}\eq_i{\pm}\eq_j\mid 1{\leq}i{<}j{\leq}n\}
\end{equation*}
for an orthonormal basis $\eq_1,\hdots,\eq_n$ of $\R^n.$ The grading of $\gt$ could also have
been obtained from
the cross marked Dynkin diagram 
 (see \cite{Bou68})
\begin{equation*}
  \xymatrix@R=-.3pc{   \alphaup_1 &  &\alphaup_{n-1} &\alphaup_{n{-}1} &
  \alphaup_n\\
\!\!\medcirc\!\! \ar@{-}[r]& \cdots \ar@{-}[r]& \!\!\medcirc\!\!\ar@{-}[r]  &\!\!\medcirc\!\!
\ar@{=>}[r]&\!\!\medcirc\\
&&&& \times}
\end{equation*}
with 
\begin{equation*}
  \alphaup_i=\eq_{i}-{\eq}_{i{-}1},\;\text{for}\, 1{\leq}i{\leq}n{-}1,\;\;\alphaup_n=\eq_n.
\end{equation*}
by setting $\deg(\alphaup_n){=}1,$ $\deg(\alphaup_i){=}0$ for $1{\leq}i{\leq}n{-}1.$ 
\par \smallskip
For $n{=}2,$ the summands $\ft_{{-}\pq}(\C^2)$ (for $\pq{\geq}1$)
are all irreducible and isomorphic either to the trivial one-dimensional representation on $\Lambda^2(\C^2)\simeq\C,$ 
for $\pq$ even, or to the two-dimensional standard representation $\Lambda^1(\C^2){=}\C^2$ for
$\pq$ odd. Therefore
we will consider
the fundamental graded Lie algebra of the third kind
$\mt=\C^2{\oplus}\Lambda^2(\C^2){\oplus}\ft_{3}(\C^2),$
with $\ft_{3}(\C^2){\simeq}\C^2.$ 
Let
$\omegaup{=}\tfrac{1}{2}(\eq_2{-}\eq_1)$ 
and take
the marker $\epi=\tfrac{1}{2\sqrt{3}}(\eq_1{+}\eq_2).$ Then 
\begin{equation*}
 \|\,\omegaup{+}\epi\,\|^2=\frac{2}{3},\;\; \|\,2\epi\,\|^2
 =\frac{2}{3},\;\; \|\,\omegaup{+}3\epi\,\|^2=2.
\end{equation*}
By setting 
\begin{equation*}
 \Rad={\bigcup}_{\pq{=}{-}3}^3\Rad_{\;\pq},\;\;\;\text{with}\;\;\; 
\begin{cases}
 \Rad_{\;0}=\{{\pm}2\omegaup,\},\\
 \Rad_{\;{\pm}1}=\{{\mp}\epi+\omegaup,{\mp}\epi-\omegaup\},\\
 \Rad_{\;{\pm}2}=\{{\mp}2\epi\},\\
 \Rad_{\;{\pm}3}=\{{\mp}3\epi{+}\omegaup,{\mp}3\epi{-}\omegaup\},
\end{cases}
\end{equation*}
we obtain a root system of type $\mathbf{G}_2.$ 
With a $2$-dimensional Cartan subalgebra $\hg_2,$
the graded Lie algebra 
\begin{equation*}
 {\sum}_{\pq{=}{-}3}^3\gt_{\pq},\;\; \text{with}\;\; 
 \gt_0{=}\hg_2{\oplus}\langle\Rad_{\;0}\rangle \;\text{and}\;
\gt_\pq{=}\langle\Rad_{\;\pq}\rangle \;\text{if $\pq{\neq}0,$}
\end{equation*}
is an effective prolongation  
of an fundamental graded Lie algebra  
of the third kind. It is the maximal one. Indeed, by 
\cite[Thm.5.3]{MMN2018} 
the effective prolongations 
of $\mt{=}{\sum}_{\pq{<}0}\gt_{\pq}$ 
are finite dimensional,
because in this case $W{=}\{0\}.$ 
Then, since $\gt$ is simple, it is maximal by 
Proposition~\ref{prop-7.3}. 
The cross marked Dynkin diagram is
\begin{equation*}
  \xymatrix@R=-.3pc{  
\alphaup_1 & \alphaup_2  \\
\!\!\medcirc\!\!\!
\ar@3{->}[r]&\!\!\medcirc\!\! \\
&\times}\quad \qquad\begin{matrix} 
\qquad\quad \\
\alphaup_1{=}2\omegaup,\;\; \alphaup_2{=}\omegaup{-}\epi,
\end{matrix}
\end{equation*}
with $\deg(\alphaup_1){=}0,$ $\deg(\alphaup_2){=}1$  
see e.g. \cite{Cartan1910,golubitsky1971,Wil1971}). 
\par\smallskip
This gradation and those introduced for $\mathbf{B}_n$ 
are compatible only with the split real forms of the complex
simple Lie algebras considered above: we obtain structures of 
effective prolongation of type $\gl_n(\R)$ 
and second kind on  
$\ot(n,n{+}1)$ for $n{\geq}3$ 
and of the third kind on the split real form of~$\mathbf{G}_2$ for $n{=}2.$\par 

\subsubsection{$\gt_{{-}1}{\simeq}V_{2\omegaup_1},V_{3\omegaup_1}
:$ construction of $\mathbf{C}_n$ and $\mathbf{G}_2$}
Let us consider now the irreducible faithful representation of $\slt_n(\C)$ on the
space $\Symm_2(\C^n)$ of degree two symmetric tensors. 
From \cite[\S{3.1}]{MMN2018}
we know
that the maximal effective prolongation
of type $\gl_n(\C)$ of $\Symm_2(\C^n)$ is finite dimensional if $n{\geq}3.$ In fact, 
it is isomorphic to a Lie algebra of type $\mathbf{C}_n.$ 
The dominant weight of $\Symm_2(\C^n)$ is $2\omegaup_1$ and we have 
$\|\,2\omegaup_1\,\|^2=4{-}\frac{4}{n}.$ We 
consider the extremal weight $\omegaup{=}2\eq_n{-}\tfrac{2}{n}\eq_0$ and
take the marker $\epi{=}\frac{2}{n}\eq_0.$ Set 
\begin{equation*} 
\begin{cases}
 \Rad_{\;{-}1}=\{\wq{+}\epi\mid \wq{\in}
 \pes(2\omegaup_1)\}=\{(\eq_i{+}\eq_j)\mid
 1{\leq}i{\leq}j{\leq}n\},\\
 \Rad_{\;0} =\{\eq_i{-}\eq_j\mid 1{\leq}i{\neq}j{\leq}n\},\\
  \Rad_{\;1}=\{\wq{-}\epi\mid \wq{\in}\pes(2\omegaup_{n{-}1})\}=\{-(\eq_i{+}\eq_j)\mid
 1{\leq}i{\leq}j{\leq}n\}.
\end{cases}
\end{equation*}
Then $\Rad={\bigcup}_{\pq{=}{-}1}^1\Rad_{\;\pq}$ is a root system of type $\mathbf{C}_n.$ 
With a Cartan subalgebra $\hg_n$ of dimension $n,$ we obtain the maximal EPFGA of type
$\gl_n(\C)$ of $\Symm_2(\C^n)$ in the form 
\begin{equation*}
 \gt={\sum}_{\pq{=}{-}1}^1\gt_{\pq}\simeq\spt(n,\C),\;\;\;\text{with}\;\;\gt_0{=}\hg_n{\oplus}\langle\Rad_{\;0}\rangle,
 \;\; \gt_{\pq}=\langle\Rad_{\;\pq}\rangle,\;\text{for $\pq{\neq}0.$}
\end{equation*}
Indeed,   
\begin{equation*}
 \Rad=\{{\pm}\eq_i{\pm}\eq_j\mid 1{\leq}i{<}j\leq{n}\}\cup\{2\eq_i\mid 1{\leq}i{\leq}n\}
\end{equation*}
is 
the set of roots of $\spt(n,\C)$ 
and we can obtain the gradation above from 
its cross marked Dynkin diagram 
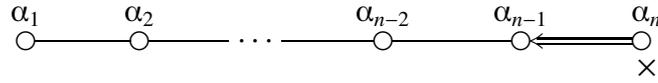
\begin{equation*}
  \xymatrix@R=-.3pc{ 
  \alphaup_1 & \alphaup_2 & &\alphaup_{n-2} 
  &\alphaup_{n-1} & \alphaup_n\\
\!\!\medcirc\!\! \ar@{-}[r]& \!\!\medcirc\!\! \ar@{-}[r]&
\cdots \ar@{-}[r]& \!\!\medcirc\!\!\ar@{-}[r]  &\!\!\medcirc\!\!
\ar@{<=}[r] & \!\!\medcirc
\\ &&&& & \times}
\end{equation*}
with 
$
 \alphaup_i=\eq_{i}-{\eq}_{i+1},\;\text{for}\, 
 1{\leq}i{\leq}n{-}1,\;\;\alphaup_n=2\eq_n,
$
by requiring that $\deg(\alphaup_i){=}0$ 
for $1{\leq}i{\leq}n{-}1$ and $\deg(\alphaup_n){=}1.$ 

\par
\smallskip
For $n{=}2,$ we consider the irreducible 
$\slt_2(\C)$-module $\Symm_3(\C^2){\simeq}
V_{3\omegaup_1}.$  
Its exterior square 
contains the irreducible representation $V_{0} \simeq
\Lambda^2(\C^2) \simeq \C.$  One can check, by using 
\cite[\S{3},\S{4},\S{6}]{MMN2018}
that the maximal effective prolongation of the 
fundamental graded Lie algebra 
 $\mt=\gt_{{-}1}\oplus\gt_{{-}2},$
with $\gt_{{-}1}=\Symm_3(\R^2)$ and $\gt_{{-}2}=\Lambda^2(\C^2),$ 
 is
finite dimensional. 
It is in fact a simple Lie algebra of type $\mathbf{G}_2.$ 
With the marker $\epi=\frac{\sqrt{3}}{2}\eq_0$ we obtain 
\begin{equation*}
 \|\pm 3\omegaup_1{+}\epi\|^2=6,\;\; \|\pm\omegaup_1{+}\,\epi\,\|^2
 =2,\;\;\|\,2\epi\,\|^2=6.
\end{equation*}
Set 
\begin{equation*} 
\begin{cases}
 \Rad_{\;{-}2}=\{2\epi\},\\
 \Rad_{\;{-}1}=\{\epi{\pm}\omegaup_1,\epi{\pm}3\omegaup_1\},\\ 
 \Rad_{\;0}=\{\pm2\omegaup_1\},\\
 \Rad_{\;1}=\{{-}\epi{\pm}\omegaup_1,{-}\epi{\pm}3\omegaup_1\},\\  
 \Rad_{\;2}=\{{-}2\epi\}.
\end{cases}
\end{equation*}
The union $\Rad={\bigcup}_{\pq{=}{-}2}^2\Rad_{\;\pq}$ is a root system of type $\mathbf{G}_2$ 
 and the prolongation 
\begin{equation*}
 \gt={\sum}_{\pq{=}{-}2}^2\gt_{\pq},\;\;\text{with}\;\; \gt_0
 {=}\hg_2{\oplus}\langle\Rad_{\;0}\rangle,\;\;
 \gt_{\pq}=\langle\Rad_{\;\pq}\rangle,
\end{equation*}
with $\hg_2$ a Cartan subalgebra of dimension $2,$ is by Proposition~\ref{prop-7.3},
a maximal effective prolongation of $\mt,$ because
$\gt$ is simple. The corresponding cross-marked Dynkin diagram is
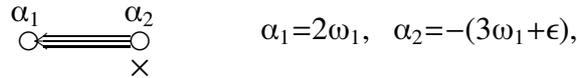
\begin{equation*}
  \xymatrix@R=-.3pc{  
\alphaup_1 & \alphaup_2  \\
\!\!\medcirc\!\!\!
\ar@3{<-}[r]&\!\!\medcirc\!\! \\
&\times}\quad \qquad\begin{matrix} 
\qquad\quad \\
\alphaup_1{=}2\omegaup_1,\;\; \alphaup_2{=}{-}(3\omegaup_1{+}\epi),
\end{matrix}
\end{equation*}
\par 
The discussion above 
applies to the split real form $\gl_{n}(\R)$ of $\gl_n(\C),$ exhibiting
$\spt(n,\R)$ and the real split form of $\mathbf{G}_2$  
as (EPGFLA)'s of 
a fundamental graded Lie algebra with structure algebra
$\gl_n(\R).$ 
\subsubsection{$\gt_{{-}1}{\simeq}V_{\omegaup_2}:$ construction of $\mathbf{D}_n$} \label{sec-7.1.3}
The next example refers to the representation  of
$\slt_n(\C)$ on ${V}_{\omegaup_2}{=}\Lambda^2(\C^n),$ 
for $n{\geq}4.$  
Since the Dynkin diagram of a simple prolongation would have a ramification node, all its roots
would have the same square
lenght $2$. Take the extremal weight $\omegaup{=}\eq_{n{-}1}{+}\eq_n{-}\tfrac{2}{n}\eq_0$
and
the marker 
$\epi{=}\frac{2}{n}\eq_0$; then  
$\|\epi{+}\wq\|^2{=}2,$ for all 
$\wq{\in}\Lambda(\omegaup_2).$  
We set $\mt{=}\gt_{{-}1}{=}\Lambda^2(\C^n)$ and 
\begin{equation*} 
\begin{cases}
 \Rad_{\;{-}1}=\{\omegaup{+}\epi\mid \omegaup{\in}\pes(\omegaup_2)\}=\{\eq_i{+}\eq_j\mid
 1{\leq}i{<}j{\leq}n\},\\
 \Rad_{\;0} =\{\eq_i{-}\eq_j\mid 1{\leq}i{\neq}j{\leq}n\},\\
  \Rad_{\;1}=\{\omegaup{-}\epi\mid \omegaup{\in}\pes(\omegaup_{n{-}2})\}=\{-(\eq_i{+}\eq_j)\mid
 1{\leq}i{<}j{\leq}n\}.
\end{cases}
\end{equation*}
The union $\Rad={\bigcup}_{\pq{=}{-}2}^2\Rad_{\;\pq}$ is a root system of type $\mathbf{D}_n$ 
 and the prolongation 
\begin{equation*}
 \gt={\sum}_{\pq{=}{-}2}^2\gt_{\pq}\simeq\ot(2n,\C),
 \;\;\text{with}\;\; \gt_0
 {=}\hg_{n}{\oplus}\langle\Rad_{\;0}\rangle,\;\;
 \gt_{\pq}=\langle\Rad_{\;\pq}\rangle,
\end{equation*}
with $\hg_n$ a Cartan subalgebra of dimension $n,$ 
is, by Proposition~\ref{prop-7.3},
a maximal effective prolongation of $\mt,$ because
$\gt$ is simple.
Indeed, 
\begin{equation*}
 \Rad=\{{\pm}\eq_i{\pm}\eq_j\mid 1{\leq}i{<}j{\leq}n\},
\end{equation*}
is  the root system of $\ot(2n,\C),$ 
with cross-marked Dynkin diagram 
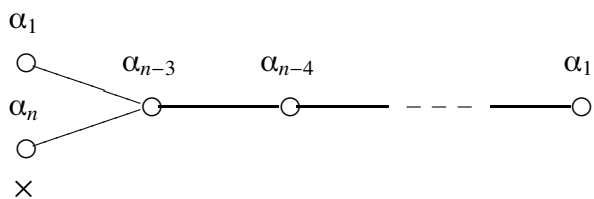
\begin{equation*}
 \xymatrix@R=.1pc{
 \alphaup_1\\
\!\!
\medcirc\!\!\! \ar@{-}[rd] &\alphaup_{n{-}3} &\alphaup_{n{-}4} &&& \alphaup_{1}\\
\alphaup_n & \!\!\medcirc\!\!\! \ar@{-}[r]  & \!\!\medcirc\!\!\! \ar@{-}[r] & \ar@{--}[r] &{}
\ar@{-}[r] & \!\!\medcirc\!\!\!
\\
\!\!
\medcirc\!\!\! \ar@{-}[ru]
\\
 \times}
\end{equation*}
with $\alphaup_i=\eq_i{-}\eq_{i{+}1}$ for $1{\leq}i{\leq}n{-}1$ and $\alphaup_n{=}(\eq_{n{-}1}{+}\eq_n).$ 
By setting $\deg(\eq_i)=\tfrac{1}{2}$ for $1{\leq}i{\leq}n$ we obtain $\deg(\alphaup_i){=}0$
for $1{\leq}i{\leq}n{-}1,$ $\deg(\alphaup_n){=}1$ and hence
the gradation above for $\gt.$ 

The gradation above is compatible both with the Satake diagram
of the split form $\gl_n(\R),$ yielding $\ot(n,n),$ 
and, for $n{=}2m$ even, also with $\slt_m(\Hb),$ which has the Satake diagram
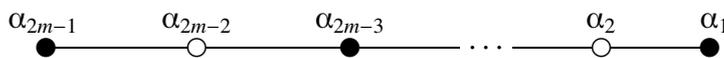
\begin{equation*}
  \xymatrix@R=-.3pc{  \alphaup_{2m{-}1} & \alphaup_{2m{-}2} & \alphaup_{2m{-}3} & 
  &\alphaup_{2} &\alphaup_{1} \\
\!\!\medbullet\!\!\!
\ar@{-}[r]&\!\!\medcirc\!\! \ar@{-}[r]& 
\!\!\medbullet\!\!\ar@{-}[r]& \cdots \ar@{-}[r]& \!\!\medcirc\!\!
\ar@{-}[r]  &\!\!\medbullet}
\end{equation*}
In this case the corresponding
  maximal effective prolongation of type $\gl_n(\R)$ of the fundamental graded Lie algebra
of the first kind $V{\simeq}\Lambda^2(\R^{2m})$ is 
isomorphic to
$\ot^*(2m),$ having cross-marked Satake diagram  
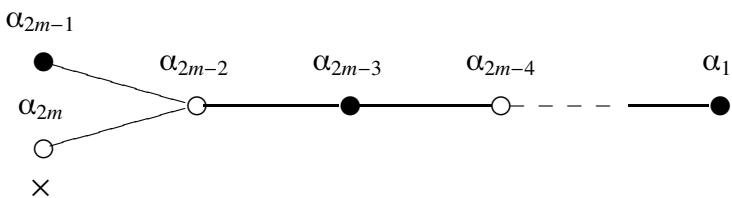
\begin{equation*}
 \xymatrix@R=.1pc{
 \alphaup_{2m{-}1}\\
\!\!\medbullet\!\!\! \ar@{-}[rd] &\alphaup_{2m{-}2} &\alphaup_{2m{-}3} 
&\alphaup_{2m{-}4} && \alphaup_{1}\\
\alphaup_{2m}
& \!\!\medcirc\!\!\! \ar@{-}[r]  & \!\!\medbullet\!\!\! \ar@{-}[r] &
\!\!\medcirc\!\! \ar@{--}[r] &{}
\ar@{-}[r] & \!\!\medbullet\!\!\!
\\
\!\!\medcirc\!\!\! \ar@{-}[ru]
\\
\times
 }
\end{equation*}
\subsubsection{$\gt_{{-}1}{\simeq}{V}_{\omegaup_3}:$ construction of 
$\mathbf{E}_6,$ $\mathbf{E}_7,$ $\mathbf{E}_8$} 
Let us consider effective prolongations 
which are constructed on  the representation of
$\slt_n(\C)$ on $\Lambda^3(\C^n),$ for $n{\geq}6.$ 
We know from \cite[\S{3.1},\S{6}]{MMN2018}
that they 
are finite dimensional. The dominant weight for  $\Lambda^3(\C^n)$ is
$\omegaup_3=\eq_1{+}\eq_2{+}\eq_3{-}\frac{3}{n}\eq_{\,0}.$
By Remark~\ref{rmk7.4} we know that the necessary and sufficient condition for
the existence of a semisimple effective prolongation $\gt$ of an $\mt$ with $\gt_{{-}1}{\simeq}\Lambda^3(\C^n)$ 
is that 
\begin{equation*}
 \|\omegaup_3\|^2=3-\frac{9}{n}<2.
\end{equation*}
Thus the only possible choices are
$n{=}6,7,8$  with corresponding markers  
\begin{equation*}
 \epi_6=\frac{1}{2\sqrt{3}}\,\eq_0,\;\; \epi_7=\frac{\sqrt{2}}{7}\eq_0,\;\; \epi_8=\frac{1}{8}\,\eq_0.
\end{equation*}
Let us set
\begin{gather*}
\begin{cases}
\Rad^{(6)}_{\;{-}2}{=}\{2\epsilonup_6\},\\
\Rad^{(6)}_{\;{-}1}{=}\{\epsilonup_6{+}\omegaup\mid\omegaup{\in}
\pes(\omegaup_3)\},
\\
\Rad^{(6)}_{\;0}{=}\left\{\eq_i{-}\eq_j\mid 1{\leq}i{\neq}j{\leq}6
\right\},\\
\Rad^{(6)}_{\;1}{=}\{-\epsilonup_6{+}\omegaup\mid
\omegaup{\in}\pes(\omegaup_{n{-}3})\},
\\
\Rad^{(6)}_{\;2}{=}\{-2\epsilonup_6\},
\end{cases}
\qquad 
\begin{cases}
\Rad^{(7)}_{\;{-}2}{=}\{2\epsilonup_7{+}\omegaup\mid\omegaup
{\in}\pes(\omegaup_6)\},
\\
\Rad^{(7)}_{\;{-}1}{=}\{\epsilonup_7{+}\omegaup
\mid 
\omegaup{\in}\pes(\omegaup_3)\},\\
\Rad^{(7)}_{\;0}{=}\left\{\eq_i{-}\eq_j\mid 1{\leq}i{\neq}j{\leq}7
\right\},\\
\Rad^{(7)}_{\;1}{=}\{{-}\epsilonup_7{+}\omegaup\mid
\omegaup{\in}\pes(\omegaup_4)\},\\
\Rad^{(7)}_{\;2}{=}\{{-}2\epsilonup_7{+}\omegaup\mid
\omegaup{\in}\pes(\omegaup_1)\},
\end{cases}
 \\
\begin{cases}
\Rad^{(8)}_{\;{-}2}{=}\{2\epsilonup_8{+}\omegaup\mid\omegaup
{\in}\pes(\omegaup_6)\},
\\
\Rad^{(8)}_{\;{-}1}{=}\{\epsilonup_8{+}\omegaup
\mid 
\omegaup{\in}\pes(\omegaup_3)\},\\
\Rad^{(8)}_{\;0}{=}\left\{\eq_i{-}\eq_j\mid 1{\leq}i{\neq}j{\leq}8
\right\},\\
\Rad^{(8)}_{\;1}{=}\{{-}\epsilonup_8{+}\omegaup\mid
\omegaup{\in}\pes(\omegaup_5)\},\\
\Rad^{(8)}_{\;2}{=}\{{-}2\epsilonup_8{+}\omegaup\mid
\omegaup{\in}\pes(\omegaup_2)\}.
\end{cases}
\end{gather*}
\par 
Then one can show that,
for each $n{=}6,7,8,$ the sets 
$\Rad^{(n)}{=}{\bigcup}_{\pq{=}{-}2}^2\Rad^{(n)}_{\pq}$
are root systems of type $\mathbf{E}_n,$ with Dynkin diagrams
{\small
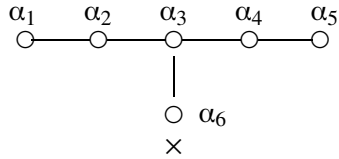
\begin{equation*}
  \xymatrix@C=.9pc@R=-.3pc{  \alphaup_1 & \alphaup_2 &\alphaup_3 
  &\alphaup_4 &\alphaup_5\\
\!\!\medcirc\!\!\!
\ar@{-}[r]&\!\!\medcirc\!\! \ar@{-}[r]
&\!\!\medcirc\!\!\ar@{-}[r]  \ar@{-}[dddddd]
&\!\!\medcirc\!\!\ar@{-}[r]  &\!\!\medcirc& && \alphaup_i{=}\eq_i{-}\eq_{i{+}1}\\
&&  \\ && \\ && \\ && \\ && \\
& & \!\medcirc\! &\!\!\!\! \!\!\!\!\!\!\!\!\!\!\!\alphaup_6 &&&& \alphaup_6{=}{-}\epsilonup_6{+}\tfrac{1}{2}(
\eq_4{+}\eq_5{+}\eq_6
{-}\eq_1{-}\eq_2{-}\eq_3)\\
& & \times}
\end{equation*}
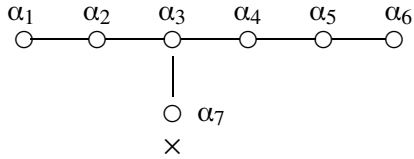
\vspace{5pt}
 \begin{equation*}
  \xymatrix@C=.9pc@R=-.3pc{  
  \alphaup_1 & \alphaup_2 &\alphaup_3 &\alphaup_4 &\alphaup_5
  &\alphaup_6\\
\!\!\medcirc\!\!\!
\ar@{-}[r]&\!\!\medcirc\!\! \ar@{-}[r]
&\!\!\medcirc\!\!\ar@{-}[r]  \ar@{-}[dddddd]
&\!\!\medcirc\!\!\ar@{-}[r] 
 &\!\!\medcirc\!\!\ar@{-}[r] &\!\!\medcirc&  
 \alphaup_i{=}\eq_i{-}\eq_{i{+}1}
 \\
&&  \\ && \\ && \\ && \\ && \\
& & \!\medcirc\! &\!\!\!\! \!\!\!\!\!\!\!\!\!\!\!\alphaup_7&&& \alphaup_7{=}{-}\epi_7{+}\tfrac{1}{7}[3
(\eq_1{+}\cdots{+}\eq_7){-}7
(\eq_1{+}\eq_2{+}\eq_3)]\\
& & \times}
\end{equation*}
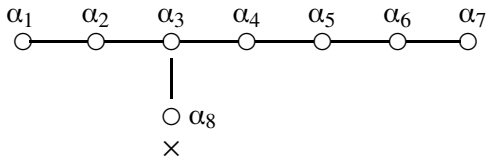
\vspace{5pt}
\begin{equation*}
  \xymatrix@C=.9pc@R=-.3pc{  \alphaup_1 & \alphaup_2 
  &\alphaup_3 &\alphaup_4 &\alphaup_5
  &\alphaup_6&\alphaup_7\\
\!\!\medcirc\!\!\!
\ar@{-}[r]&\!\!\medcirc\!\! \ar@{-}[r]&\!\!\medcirc\!\!\ar@{-}[r]  \ar@{-}[dddddd]
&\!\!\medcirc\!\!\ar@{-}[r]  &\!\!\medcirc\!\!\ar@{-}[r] &\!\!\medcirc\!\!\ar@{-}[r] &\!\!\medcirc
&\alphaup_i{=}\eq_i{-}\eq_{i{+}1}\\
&&  \\ && \\ && \\ && \\ && \\
& & \!\medcirc\! 
&\! \!\!\!\!\!\!\!\!
\!\!\!\!\!
\!\!\!\!\alphaup_8
&&&&\alphaup_8{=}{-}\epsilonup_8{+}\tfrac{1}{8}[
8(\eq_i{+}\eq_j{+}\eq_k){-}3(\eq_1{+}\cdots{+}\eq_8)]\\
& & \times
}
\end{equation*}}
Accordingly, we obtain on the complex Lie algebras of
type  $\mathbf{E}_6,$ $\mathbf{E}_7,$ $\mathbf{E}_8$ 
structures of effective prolongations of type $\gl_n(\C),$ 
for $6{\leq}n{\leq}8,$ by setting 
\begin{align*} \gt={\sum}_{\pq{=}{-}2}^2\gt_{\pq},\quad\text{with}\quad
\begin{cases}
\gt_{\pq}{=}\langle \Rad^{(n)}_{\pq}\rangle, \quad 
\text{for $\pq{=}{\pm}{1},{\pm}{2},$}\\
\gt_0{=}\hg_n\oplus
\langle\Rad^{(n)}_0\rangle \simeq\gl_n(\C),
\end{cases}
\end{align*}
where $\hg_n$ is an $n$-dimensional 
Cartan subalgebra. \par 
In all cases we have $\gt_{{-}1}{\simeq}\Lambda^3(\C^n)$
and $\gt_{{-}2}{\simeq}\Lambda^6(\C^n),$ with the Lie
brackets defined by the exterior product.
This gradation is compatible with the non compact 
real forms $\mathbf{E}\mathrm{I},\mathrm{I\!{I}},
\mathrm{I\!{I}\!{I}}$ 
of $\mathbf{E}_6$
(see below). In these cases 
$\gt_{{-}1}{\simeq}\Lambda^3(\R^6),$ while $[\gt_0,\gt_0]$
is simple and of type $\slt_6(\R),$ $\su(3,3)$ and 
$\su(1,4),$ respectively.
For $n{=}7,8,$ the gradation is only consistent with the real split forms
of $\mathbf{E}_7$ and $\mathbf{E}_8$
yielding the real analogue of the examples 
above.
\subsection{Structure algebras of type $\mathbf{B}$}
The root system of a complex Lie algebra of type
$\mathbf{B}_m$ (isomorphic to $\ot(2n{+}1,\C)$) 
is 
\begin{equation*}
 \Rad=\{{\pm}\eq_i\mid 1{\leq}i{\leq}m\}\cup\{{\pm}\eq_i{\pm}\eq_j\mid 1{\leq}i{<}j{\leq}m\}
\end{equation*}
for an orthonormal basis $\eq_1,\hdots,\eq_m$ of $\R^m$ 
and  its Dynkin diagram 
\begin{equation*}
  \xymatrix@R=-.3pc{  \alphaup_1   &\alphaup_{2} & &\alphaup_{m-1}& \alphaup_m  \\
\!\!\medcirc\!\!\!
\ar@{-}[r]&\!\!\medcirc\!\! \ar@{-}[r]& \cdots \ar@{-}[r]& \!\!\medcirc\!\!\ar@2{->}[r]  &\!\!\medcirc}
\end{equation*}
has simple roots that 
can be chosen to be $\alphaup_i=\eq_{i}{-}\eq_{i{+}1}$ for $1{\leq}i{<}m$ and $\alphaup_m{=}\eq_m.$
Its fundamental weights are
$$\sigmaup_1{=}\eq_1,
\; \sigmaup_2{=}\eq_1{+}\eq_2,\; \hdots,
\sigmaup_{m{-}1}{=}\eq_1{+}\cdots{+}\eq_{m{-}1},\;
\sigmaup_m{=} \tfrac{1}{2}(\eq_1{+}
\cdots{+}\eq_m).$$
Its irreducible representation with maximal weight 
$\sigmaup_m$
is called its \emph{complex spin representation} and indicated by $S_{\!2m+1}^{\C}.$ 
 Its weights  
$$\pes(\sigmaup_m)
{=}\{\tfrac{1}{2}({\pm}\eq_1{\pm}\cdots{\pm}\eq_m)\}$$ 
are all simple, so that $\dim(S_{2m+1}^{\C}){=}2^m.$ 
\par 
We know (see \cite[Ex.3.3,3.8]{MMN2018})
that all the effective prolongations 
of type $\ot(n,\C),$ with $n{\geq}2,$ 
or $\mathfrak{co}(n,\C), $ with $n{\geq}3,$
of a fundamental graded Lie algebra of the first kind
are finite dimensional. Then this holds also 
for fundamental graded Lie algebras of any finite kind.
\subsubsection{Spin representation for $\mathbf{B}_m$}
Let us take an $\mt$ with 
$\gt_{{-}1}$ equal to the spin representation 
$S{=}S_{2m{+}1}^{\C}$ of $\ot(2m{+}1,\C).$ 
They will be complexifications of 
\textit{real} spin representations of real orthogonal
Lie algebras.
\par
To obtain a semisimple effective prolongation of type $\co(2m{+}1,\C)$ of a 
fundamental graded Lie algebra with
$\gt_{{-}1}{=}S,$ since its dominant weight 
is attached to a simple root of length $1,$ by \eqref{eq7.2}
it is necessary to produce a \textit{new 
 lenght $1$ root} by adding a \textit{marker}
to the dominant weight 
$\sigmaup_m$ of $S.$ Since $\|\sigma_m\|^2=\frac{m}{4},$
this is possible iff $m{\leq}3.$ 
\par 
For 
$m{=}1,$ we have $\ot(3,\C){\simeq}\slt_2(\C)$ and hence we refer to \S\ref{sub7.1.1}
(in particular, we may consider the split real form of $\mathbf{G}_2$ as related to the 
real spin
representation of $\ot(1,2)$).
\par For $m{=}2$ we obtain  a maximal effective prolongation 
which is isomorphic to 
$\spt(3,\C)$
and whose real form $\spt_{1,2}$ can be associated to the real spin
representation of $\ot(2,3).$ 
\par 
When $m{=}3$ 
we obtain 
a presentation of the exceptional Lie algebra of
type 
$\mathbf{F}_4,$
as an
effective prolongation of type $\co(7,\C)$ 
of a fundamental graded Lie algebra 
of depth $2$ with
$\gt_{{-}1}$ equals to the $8$-dimensional spin representation
of $\ot(7,\C).$
Taking 
$\epi$ orthogonal to $\langle\eq_1,\eq_2,\eq_3\rangle$
and of lenght $\frac{1}{2},$ we obtain 
\begin{equation*}
\|\,\sigmaup_3{+}\epi\,\|^2=1,\quad \|\,\eq_1{+}2\epi\,\|^2=2,\;\;
\|\,2\epi\,\|^2{=}2.
\end{equation*}
We note that the vector representation 
$V_{\eq_1}\simeq\C^7$ of $\ot(7,\C)$ 
is an irreducible summand
of $\Lambda^2(S).$ Then we can take $\mt{=}S{\oplus}V_{\eq_1}.$
 Set
 \begin{align*}
 \begin{cases}
 \Rad_{\;{-}2}=\{2\epi+\omegaup\mid \omegaup{\in}\pes(\eq_1)\}=
 \{2\epi\}\cup\{2\epi{\pm}\eq_i\mid 1{\leq}i{\leq}3\},\\
 \Rad_{\;{-}1}=\{\epi+\omegaup\mid\omegaup{\in}\pes(\sigmaup_3)\}
 =\{\epi{+}\tfrac{1}{2}({\pm}\eq_1{\pm}\eq_2{\pm}\eq_3)
 \},\\
 \Rad_{\;0}=\{{\pm}\eq_i\mid 1{\leq}i{\leq}3\}
 \cup\{{\pm}\eq_i{\pm}
 \eq_j\mid 1{\leq}i{<}j{\leq}3\},\\
 \Rad_{\;1}=\{{-}\epi+\omegaup\mid\omegaup{\in}\pes(\sigmaup_3)\}
 =\{-\epi{+}\tfrac{1}{2}({\pm}\eq_1{\pm}\eq_2{\pm}\eq_3)
 \},\\ 
 \Rad_{\;2}=\{{-}2\epi+\omegaup\mid \omegaup{\in}\pes(\eq_1)\}=
 \{{-}2\epi\}\cup\{{-}2\epi{\pm}\eq_i\mid 1{\leq}i{\leq}3\}.
 \end{cases}
 \end{align*}
Then $\Rad={\bigcup}_{\pq{=}{-}2}^2\Rad_{\pq}$ 
 is a root system of type 
 $\mathbf{F}_4.$ \par 
With the four dimensional Cartan algebra $\hg_4$ 
of $\mathfrak{co}(7,\C),$ 
we obtain an effective prolongation of $\mt$ in the form 
\begin{equation*}
 \gt{=}{\sum}_{\pq{=}{-}2}^2\gt_{\pq}, \;\;\text{with}\;\; 
 \gt_{\pq} = 
\begin{cases}
 \langle\Rad_{\;\pq}\rangle ,\quad\text{for}\; \pq{=}\pm{1},{\pm}2,\\
 \hg_4\oplus  \langle\Rad_{\;0}\rangle\simeq
 \mathfrak{co}(7,\C),
 \;\; \text{for}\; \pq{=}0.
\end{cases}
\end{equation*}
\par

There are two non compact real forms of $\mathbf{F}_4,$ 
with Satake diagrams \par 
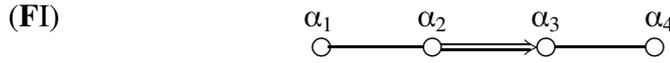
\begin{equation}\tag{$\mathbf{F}\mathrm{I}$}
  \xymatrix@R=-.3pc{  \alphaup_1 & \alphaup_2 &\alphaup_3 
  &\alphaup_4 \\
\!\!\medcirc\!\!\!
\ar@{-}[r]&\!\!\medcirc\!\! \ar@2{->}[r]
&\!\!\medcirc\!\!\ar@{-}[r]  &\!\!\medcirc}
\end{equation}
\begin{equation}\tag{$\mathbf{F}\mathrm{I\!{I}}$}
  
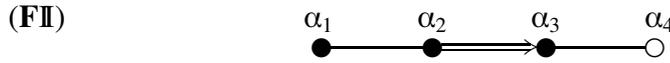
\xymatrix@R=-.3pc{  \alphaup_1 & \alphaup_2 
  &\alphaup_3 &\alphaup_4 \\
\!\!\medbullet\!\!\!
\ar@{-}[r]&\!\!\medbullet\!\! \ar@2{->}[r]
&\!\!\medbullet\!\!\ar@{-}[r]  &\!\!\medcirc}
\end{equation}
They correspond to the \textit{real} spin representations for a 
$7$-di\-men\-sional
vector representation. This means that we have 
to take the orthogonal algebras 
$\ot(3,4)$ and $\ot(7),$ discarding 
the other possible two, namely
$\ot(2,5)$ and $\ot(1,6),$ 
which have quaternionic spin representations 
(see e.g. \cite[p.103]{Deligne99}).
\begin{rmk}
 The lifting of the quaternionic spin representations 
 of $\ot(2,5)$ and $\ot(1,6)$ yields an irreducible representation
 of $\ot(7,\C){\oplus}\slt_2(\C)$ 
 for which we cannot comply \eqref{eq7.2}. 
 In particular, there are infinitely many 
  summands with negative indices in the
associated graded Kac-Moody algebra. \par 
 In contrast, for $m{=}2,$ 
the quaternionic spin representations of $\ot(5)$ and 
 $\ot(1,4)$ can be lifted 
 to maximal effective prolongations of
 type $\co(5,\C)\oplus\gl_2(\C)$ isomorphic to $\spt(4,\C).$ 
 The corresponding real models are $\spt_{1,3}$ and 
 $\spt_{2,2},$ respectively.
\end{rmk}
\subsection{Structure algebras of type $\mathbf{C}$} 
Let us consider the complex Lie algebra $\spt(m,\C),$ 
of type $\mathbf{C}_m.$ Its root system is 
\begin{equation*}
 \Rad=\{{\pm}2\eq_i\mid 1{\leq}i{\leq}m\}
 \cup\{{\pm}\eq_i{\pm}\eq_j\mid 1{\leq}i{<}j{\leq}m\}
\end{equation*}
for an orthonormal basis $\eq_1,\hdots,\eq_m$ of $\R^m$ 
and  its Dynkin diagram 
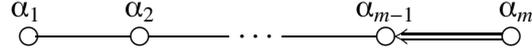
\begin{equation*}
  \xymatrix@R=-.3pc{  \alphaup_1   &\alphaup_{2} & &\alphaup_{m-1}
  & \alphaup_m  \\
\!\!\medcirc\!\!\!
\ar@{-}[r]&\!\!\medcirc\!\! \ar@{-}[r]
& \cdots \ar@{-}[r]& \!\!\medcirc\!\!\ar@2{<-}[r]  &\!\!\medcirc}
\end{equation*}
has simple roots that 
can be chosen to be $\alphaup_i{=}\eq_{i}{-}\eq_{i{+}1}$ 
for $1{\leq}i{<}m$ and $\alphaup_m{=}2\eq_m.$
Its weights lattice is the $\Z$-module $\Z^m$ in $\R^m,$
with fundamental weights $\omegaup_j{=}{\sum}_{i{=}1}^j\eq_i,$
for $j{=}1,\hdots,m.$  
We know that
the maximal effective prolongation of type $\spt(m,\C)$ 
of the fundamental graded Lie algebra 
of the first kind $\C^{2m}$
are infinite dimensional
(see \cite[Ex.3.10]{MMN2018},
or \cite[p.10]{Kob}). 
On the other hand, if we choose another faithful irreducible  
representation $V$ of $\spt(m,\C),$ 
then its maximal effective prolongation of type
$\mathfrak{csp}(m,\C)$ is finite dimensional 
by \cite[Prop.3.13]{MMN2018}.\par
Let us take for instance 
$V{=}V_{\omegaup_m}{\subset}\Lambda^m(\C^{2m}).$ 
Then $\spt(m,\C)$ acts on $V$ as an algebra of  transformations 
that keep invariant the bilinear form on $V$ that can be obtained from
 the exterior product of elements of 
$\Lambda^m(\C^{2m}).$
If $m$ is even, this is a nondegenerate symmetric bilinear form and the finite dimension of
the maximal prolongation can be also checked by using 
\cite[Ex.3.3,3.8]{MMN2018}.\par
When $m$ is odd, we can use the exterior product on $\Lambda^m(\C^{2m})$ 
to define a Lie product on $V,$ yielding a fundamental graded Lie algebra $\mt{=}V{\oplus}\Lambda^{2m}(\C^{2m})$
of the second kind, which has a finite dimensional
maximal effective prolongation of type $\mathfrak{csp}(m,\C).$ The 
fundamental weight 
$\omegaup_m$ is  attached to 
the long root $2\eq_m.$ In order to be able to find a marker 
$\epi$
to embed $\omegaup_m{+}\epi$ 
into the root system of a simple Lie algebra, we need that
$\|\omegaup\|^2{=}m{<}4,$ i.e. that
$m{\leq}3.$ Te case $m{=}3$ 
 leads to another presentation of the 
 exceptional Lie algebra $\mathbf{F}_4.$
 \subsubsection{The exceptional Lie algebra of type $\mathbf{F}_4$} 
We take the marker $\epi$ as a unit vector orthogonal to
$\eq_1,\eq_2,\eq_3.$ 
We note that 
\begin{equation*}
\|\,\omegaup_3{+}\epi\,\|^2{=}4,\;\; \|\,\eq_1{+}\epi\,\|^2=2,\;\;
\|\,2\epi\,\|^2=4 
\end{equation*} 
and  set
 \begin{align*} 
\begin{cases}
 \Rad_{\;{-}2}=\{2\epi\},\\
 \Rad_{\;{-}1}=\{\epi{+}\omegaup\mid\omegaup{\in}\pes(\omegaup_3)\}
 =
 \{\epi{\pm}\eq_1{\pm}\eq_2{\pm}\eq_3
 \}{\cup}\{\epi{\pm}\eq_i\mid 1{\leq}i{\leq}3\},\\
 \Rad_{\;\;0}=\{{\pm}2\eq_i\mid 1{\leq}i{\leq}3\}\cup\{{\pm}\eq_i{\pm}
 \eq_j\mid 1{\leq}i{<}j{\leq}3\},\\
 \Rad_{\;1}=\{{-}\epi{+}\omegaup\mid \omegaup{\in}\pes(\omegaup_3)\}
 =
 \{{-}\epi{\pm}\eq_1{\pm}\eq_2{\pm}\eq_3
 \}{\cup}\{\epi{\pm}\eq_i\mid 1{\leq}i{\leq}3\},\\ 
  \Rad_{\;2}=\{-2\epi\}.
\end{cases}
 \end{align*}
Then
\begin{equation*}
 \Rad{=}{\bigcup}_{\pq{=}{-}2}^2\Rad_{\;\pq}=\{{\pm}2{\eq}_i\mid 
 1{\leq}i{\leq}4
 \}\cup\{{\pm}\eq_i{\pm}\eq_j\mid 
 1{\leq}i{<}j{\leq}4
 \}
 \cup\left\{{\pm}\eq_1{\pm}\eq_2{\pm}\eq_3{\pm}\eq_{4})\right\},
\end{equation*}
is a root system of type $\mathbf{F}_4.$
We have, with a $4$-dimensional Cartan subalgebra $\hg,$  
 \begin{equation*}
 \gt{=}{\sum}_{\pq{=}{-}2}^2\gt_{\pq}, \;\;\text{with}\;\; 
 \gt_{\pq} = 
\begin{cases}
 \langle\Rad_{\;\pq}\rangle ,\quad\text{for}\; \pq{=}\pm{1},{\pm}2,\\
 \hg_4\oplus  \langle\Rad_{\;0}\rangle\simeq\ot(7,\C){\oplus}\C,
 \;\; \text{for}\; \pq{=}0.
\end{cases}
\end{equation*}
We have $\dim(\gt_0){=}22,$ $\dim(\gt_{{\pm}1}){=}14,$ 
$\dim(\gt_{{\pm}2})=1.$ \par 
This is the maximal  effective prolongation of a fundamental graded Lie algebra of
the second kind. 
The cross marked 
Dynkin diagram associated to $\gt$ is
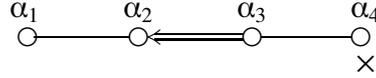
\begin{equation*}
  \xymatrix@R=-.3pc{  \alphaup_1 
  & \alphaup_2 &\alphaup_3 &\alphaup_4 \\
\!\!\medcirc\!\!\!
\ar@{-}[r]&\!\!\medcirc\!\! \ar@2{<-}[r]
&\!\!\medcirc\!\!\ar@{-}[r]  &\!\!\medcirc\\
&&& \times}
\end{equation*}
with simple roots  
$
 \alphaup_1{=}\eq_1{-}\eq_2,\;\alphaup_2{=}
 \eq_2{-}\eq_3,\;\alphaup_3{=}2\eq_3,\;\alphaup_4{=}
 \eq_4{-}\eq_{1}{-}\eq_{2}{-}\eq_3.
 $
The gradation is obtained by setting 
$\deg(\eq_i){=}0$ for $1{\leq}i{\leq}3$ and 
 $\deg(\eq_4){=}1.$ 
\par
Only the split real form $\mathbf{F}\mathrm{I}$ is compatible with this grading, 
yielding a real equivalent of the complex case.

\subsection{Structure algebras of type $\mathbf{D}$}
The diagram $\mathbf{D}_m$ (we assume $m{\geq}4$)
corresponds to the orthogonal algebra $\ot(2m,\C).$ Its root system is defined,
in an orthonormal basis $\eq_1,\hdots,\eq_m$ of $\R^m,$ by
\begin{equation*}
 \Rad(\mathbf{D}_m)=\{{\pm}\eq_i{\pm}\eq_j\mid 1{\leq}i{<}j{\leq}m\},
\end{equation*}
and we consider the corresponding 
Dynkin diagram 
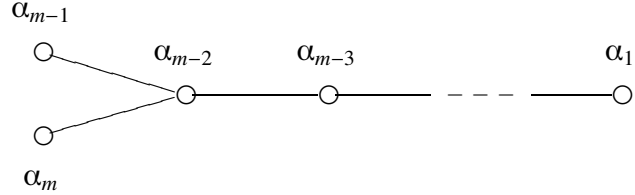
\begin{equation*}
 \xymatrix@R=.1pc{ \alphaup_{m{-}1} \\
\!\!\medcirc\!\!\! \ar@{-}[rd] &\alphaup_{m{-}2} &\alphaup_{m{-}3} &&& \alphaup_{1}\\
& \!\!\medcirc\!\!\! \ar@{-}[r]  & \!\!\medcirc\!\!\! \ar@{-}[r] & \ar@{--}[r] &{}
\ar@{-}[r] & \!\!\medcirc\!\!\!
\\
\!\!\medcirc\!\!\! \ar@{-}[ru]
\\ 
\alphaup_m\\
 }
\end{equation*}
with $\alphaup_i=\eq_i{-}\eq_{i{+}1}$ for $1{\leq}i{\leq}m{-}1$ and $\alphaup_m{=}\eq_{m{-}1}{+}\eq_m.$ 
The maximal root is $\eq_1{+}\eq_m$ and the 
fundamental weights are 
\begin{align*} 
\omegaup_j{=}{\sum}_{i{=}1}^j\eq_i,\;\;\text{for $1{\leq}j{\leq}m{-}2,$},\;\;
\omegaup_{m{-}1}{=}\tfrac{1}{2}(\eq_1{+}
\cdots{+}\eq_{m{-}1}{-}\eq_m),\qquad\\
\omegaup_{m}{=}\tfrac{1}{2}(\eq_1{+}\cdots{+}\eq_m).
\end{align*}
The last two are the dominant weights of  
two complex spin representations $S^{\C}_{m,\pm},$
with opposite chiralities and simple weights
 \begin{align*}
 \Wi(\omegaup_{m{-}1}){=}\Wi_-(m)
 &{=}\left.\left\{\tfrac{1}{2}{\sum}_{i{=}1}^m\aq_i\eq_i\right| 
 \aq_i{=}\pm{1},\; \aq_1{\cdots}\aq_m{=}{-}1\right\},\\
 \pes(\omegaup_m){=} 
  \Wi_+(m)&{=}
  \left.\left\{\tfrac{1}{2}{\sum}_{i{=}1}^m\aq_i\eq_i\right| 
  \aq_i{=}\pm{1},\; \aq_1{\cdots}\aq_m{=}1\right\}.
\end{align*}We call $V^{\C}_m{=}V_{\omegaup_1}{\simeq}\C^{2m}$ the complex vector representation.\par
\subsubsection{Real Spin representations of real Lie algebras of type $\mathbf{D}$}

Let us first consider complex 
effective prolongations with structure algebra $\ot(2m,\C)$ in which
$\gt_{{-}1}$ is a spin representation. A necessary condition for finding  
 a marker $\epi$ for which 
$\omegaup_{m{-}1}{+}\epi$
or $\omegaup_{m}{+}\epi$ 
could be embedded into the root system of a simple Lie algebra 
is that 
$\|\omegaup_{m{-}1}\|^2=\|\omegaup_m\|^2=\frac{m}{4}{<}2,$ 
i.e. that $m{=}4,5,6,7.$  \par
When $m{=}4,$ the spin and the vector representations are isomorphic 
and the maximal effective prolongation of type $\ot(8,\C)$ of the abelian Lie algebra $V^{\C}_4{\simeq}S^{\C}_{4,\pm}$ 
is just the orthogonal algebra $\ot(10,\C)$  
(see \cite[Ex.3.8]{MMN2018}).

For $m{=}5,6,7$ 
we obtain the three exceptional Lie algebras of
type $\mathbf{E}.$ 
Denote by $\epi_{m{+}1}$ a vector of $\R^{m{+}1},$ 
orthogonal to $\eq_1,\hdots,\eq_{m}$ and  
with $\|\,\epi_{m{+}1}\,\|^2=\frac{8{-}m}{4}.$ 
We define the sets 
\begin{gather*} 
\begin{cases}
 \Rad^{(6)}_{\;-1}=\{\epi_6+\wq\mid \wq{\in}\Wi_{-}(5)\},\\
 \Rad^{(6)}_{\;0}=\{{\pm}\eq_i{\pm}\eq_j\mid 1{\leq}i{<}j\leq{5}\},\\
 \Rad^{(6)}_{\;1}=\{{-}\epi_6+\wq\mid \wq{\in}\Wi_{+}(5)\} 
\end{cases}
\\
\begin{cases}
\Rad^{(7)}_{\;{-}2}=\{2\epi_7\},\\
 \Rad^{(7)}_{\;-1}=\{\epi_7{+}\wq\mid \wq{\in}\Wi_{+}(6)\},\\
 \Rad^{(7)}_{\;0}=\{{\pm}\eq_i{\pm}\eq_j\mid 1{\leq}i{<}j\leq{6}\},\\
 \Rad^{(7)}_{\;1}=\{{-}\epi_7+\wq\mid \wq{\in}\Wi_{+}(6)\},\\
 \Rad^{(7)}_{\;2}=\{{-}2\epi_7\},
\end{cases}
\qquad
\begin{cases}
\Rad^{(8)}_{\;{-}2}=\{2\epi_8{\pm}\eq_i\mid 1{\leq}i{\leq}7\},\\
 \Rad^{(8)}_{\;-1}=\{\epi_8+\wq\mid \wq{\in}\Wi_{-}(7)\},\\
 \Rad^{(8)}_{\;0}=\{{\pm}\eq_i{\pm}\eq_j\mid 1{\leq}i{<}j\leq{7}\},\\
 \Rad^{(8)}_{\;1}=\{{-}\epi_8+\wq\mid \wq{\in}\Wi_{+}(7)\},\\
 \Rad^{(8)}_{\;2}=\{{-}2\epi_8{\pm}\eq_i\mid 1{\leq}i{\leq}7\},
\end{cases}
\end{gather*}
Then $\Rad^{(6)}={\bigcup}_{\pq{=}{-}1}^1\Rad^{(6)}_{\;\pq},$
$\Rad^{(7)}={\bigcup}_{\pq{=}{-}2}^2\Rad^{(7)}_{\;\pq},$
$\Rad^{(8)}={\bigcup}_{\pq{=}{-}2}^2\Rad^{(8)}_{\;\pq},$
are root systems of type $\mathbf{E}_6,$ $\mathbf{E}_7,$ 
$\mathbf{E}_8,$ respectively and
 we obtain on the complex Lie algebras of type $\mathbf{E}_n$ structures of effective prolongation
for structure algebras of type $\mathbf{D}_{n{-}1}$ and their spin representations, of the form 
\begin{equation*}
 \gt={\sum}\gt_{\pq},\;\;\text{with}\;\; \gt_0
 {=}\mathfrak{co}(2n{-}2,\C),\;\; 
 \gt_{\pq}=\langle\Rad^{(n)}_{\;\pq}\rangle
 \;\text{for}\; \pq{\neq}0.
\end{equation*}
We note that $\gt_{{\pm}1}$ are 
for $n{=}6,8$ 
spin representations with opposite chirality,
while for $\mathbf{E}_7$ the representations $\gt_{{\pm}1}$
have equal chiralities. For $\mathbf{E}_8$ the
$\gt_{{\pm}2}$ representations are vectorial.
\par
Since the complexification of the spin representation that we found in these cases
are by construction
irreducible over $\C,$ these complex effective prolongation must be complexifications
of \textit{real} spin representations. 
We recall that the spin representations of $\ot(\pq,\qq)$ are (see e.g. \cite[p.103]{Deligne99})
\begin{equation*} 
\begin{cases}
\text{real} & \text{if}\;\; \qq{-}\pq\equiv 0,1,7\mod 8,\\
\text{complex} & \text{if}\;\; \qq{-}\pq\equiv 2,6\;\; \;\mod 8,\\
\text{quaternionic} & \text{if}\;\; \qq{-}\pq\equiv 3,4,5\mod 8.\\
\end{cases}
\end{equation*}
Thus, for the real forms of $\ot(10,\C),$ the semisimple ideal of  
$\gt_0$ can only be $\ot(5,5)$ $\ot(1,9).$ 
Corresponding, we obtain the real forms of $\mathbf{E}_6$ 
having 
cross-marked Satake diagrams
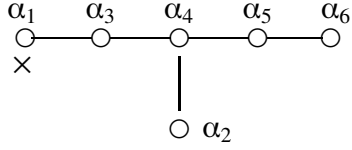
\begin{equation}\tag{$\mathbf{E}{\mathrm{I}}$}
  \xymatrix@C=.9pc@R=-.3pc{  \alphaup_1 & \alphaup_3 &\alphaup_4 &\alphaup_5 &\alphaup_6\\
\!\!\medcirc\!\!\!
\ar@{-}[r]&\!\!\medcirc\!\! \ar@{-}[r]&\!\!\medcirc\!\!\ar@{-}[r]  \ar@{-}[dddddd]
&\!\!\medcirc\!\!\ar@{-}[r]  &\!\!\medcirc\\
\times &&  \\ && \\ && \\ && \\ && \\
& & \!\medcirc\! &\!\!\!\! \!\!\!\!\!\!\!\!\!\!\!\alphaup_2}
\end{equation}
\begin{equation}\tag{$\mathbf{E}{\mathrm{I\!{V}}}$}
  
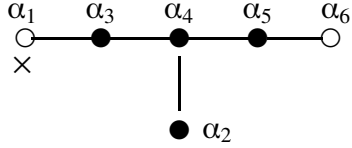
\xymatrix@C=.9pc@R=-.3pc{  \alphaup_1 & \alphaup_3 &\alphaup_4 &\alphaup_5 &\alphaup_6\\
\!\!\medcirc\!\!\!
\ar@{-}[r]&\!\!\medbullet\!\! \ar@{-}[r]&\!\!\medbullet\!\!\ar@{-}[r]  \ar@{-}[dddddd]
&\!\!\medbullet\!\!\ar@{-}[r]  &\!\!\medcirc\\
\times&&  \\ && \\ && \\ && \\ && \\
& & \!\medbullet\! &\!\!\!\! \!\!\!\!\!\!\!\!\!\!\!\alphaup_2}
\end{equation}
\par\smallskip
For $\mathbf{E}_7$ we obtain a graded algebra of the second kind. 
The real forms of $\ot(12,\C)$ having real spin representations are $\ot(6,6),$ 
$\ot(2,10)$  
and $\ot^*(12)$ (see e.g. \cite{barut}).
Thus we obtain all non compact real forms of $\mathbf{E}_7$ as (EPGFLA)'s of the real spin
representations for vector dimension $12,$ which correspond to the 
cross-marked Satake diagrams
 \begin{equation}\tag{$\mathbf{E}\mathrm{V}$}
  
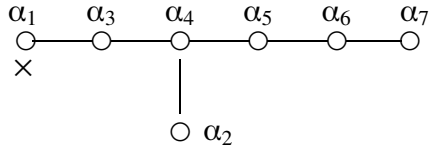
\xymatrix@C=.9pc@R=-.3pc{  \alphaup_1 & \alphaup_3 &\alphaup_4 &\alphaup_5 &\alphaup_6
  &\alphaup_7\\
\!\!\medcirc\!\!\!
\ar@{-}[r]&\!\!\medcirc\!\! \ar@{-}[r]&\!\!\medcirc\!\!\ar@{-}[r]  \ar@{-}[dddddd]
&\!\!\medcirc\!\!\ar@{-}[r]  &\!\!\medcirc\!\!\ar@{-}[r] &\!\!\medcirc\\
\times &&  \\ && \\ && \\ && \\ && \\
& & \!\medcirc\! &\!\!\!\! \!\!\!\!\!\!\!\!\!\!\!\alphaup_2}
\end{equation}

\begin{equation}\tag{$\mathbf{E}\mathrm{V\!{I}\!{I}}$}
  
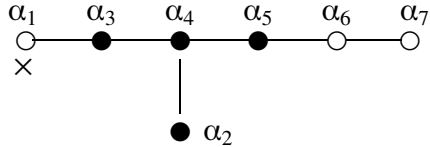
\xymatrix@C=.9pc@R=-.3pc{  \alphaup_1 & \alphaup_3 &\alphaup_4 &\alphaup_5 &\alphaup_6
  &\alphaup_7\\
\!\!\medcirc\!\!\!
\ar@{-}[r]&\!\!\medbullet\!\! \ar@{-}[r]&\!\!\medbullet\!\!\ar@{-}[r]  \ar@{-}[dddddd]
&\!\!\medbullet\!\!\ar@{-}[r]  &\!\!\medcirc\!\!\ar@{-}[r] &\!\!\medcirc\\
\times&&  \\ && \\ && \\ && \\ && \\
& & \!\medbullet\! &\!\!\!\! \!\!\!\!\!\!\!\!\!\!\!\alphaup_2}
\end{equation}
 \begin{equation}\tag{$\mathbf{E}\mathrm{V\!{I}}$}
  
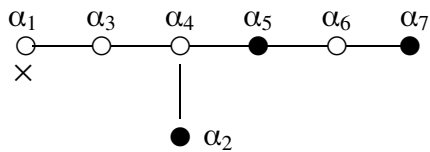
\xymatrix@C=.9pc@R=-.3pc{  \alphaup_1 & \alphaup_3 &\alphaup_4 &\alphaup_5 &\alphaup_6
  &\alphaup_7\\
\!\!\medcirc\!\!\!
\ar@{-}[r]&\!\!\medcirc\!\! \ar@{-}[r]&\!\!\medcirc\!\!\ar@{-}[r]  \ar@{-}[dddddd]
&\!\!\medbullet\!\!\ar@{-}[r]  &\!\!\medcirc\!\!\ar@{-}[r] &\!\!\medbullet\\
\times &&  \\ && \\ && \\ && \\ && \\
& & \!\medbullet\! &\!\!\!\! \!\!\!\!\!\!\!\!\!\!\!\alphaup_2}
\end{equation}
For $\mathbf{E}\mathrm{V}$ the semisimple part of the degree zero subalgebra is $\ot(6,6)$
and for $\mathbf{E}\mathrm{V\!{I}\!{I}}$ it is $\ot(2,10).$ In both cases the spin representations involved are real.
\par
We recall that the real form $\ot^*(12)$ is obtained from $\ot(12,\C)$ by  the conjugation 
which is described on the roots by 
\begin{equation*}
 \bar{\eq}_{2h-1}=\eq_{2h},\;\text{for}\; 1{\leq}h{\leq}3.
\end{equation*}
Accordingly, $\mathfrak{so}^*(12)$ has both a real and
a quaternionic spin representation.

\par\smallskip
The real forms of $\ot(14,\C)$ admitting 
real spin representations
are $\ot(7,7)$ and $\ot(3,11)$ ($\mathfrak{so}^*(14)$ has a complex spin representation).
In this case, $\gt_{{\pm}1}$ are spin representations with
opposite chiralities and  $\gt_{{\pm}2}$  dual copies of
the vector representation 
of $\ot(14,\C),$ which correspond to real vector representations $V^{\R}_7{=}\R^{14}.$ 
Thus we obtain all non compact real Lie algebras
of type $\mathbf{E}_8$ as effective prolongations of type 
$\ot(\pq,\qq)$ for FGLA's $\mt$ whose $\gt_{{-}1}$ component is a \textit{real} spin representation.
Their cross-marked Satake diagrams are
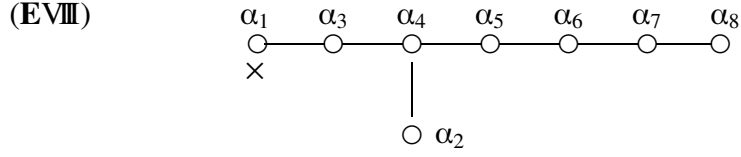
\begin{equation}\tag{$\mathbf{E}\mathrm{V\!{I}\!{I}\!{I}}$}
  \xymatrix@C=.9pc@R=-.3pc{  \alphaup_1 & \alphaup_3 &\alphaup_4 &\alphaup_5 &\alphaup_6
  &\alphaup_7&\alphaup_8\\
\!\!\medcirc\!\!\!
\ar@{-}[r]&\!\!\medcirc\!\! \ar@{-}[r]&\!\!\medcirc\!\!\ar@{-}[r]  \ar@{-}[dddddd]
&\!\!\medcirc\!\!\ar@{-}[r]  &\!\!\medcirc\!\!\ar@{-}[r] &\!\!\medcirc\!\!\ar@{-}[r] &\!\!\medcirc\\
\times &&  \\ && \\ && \\ && \\ && \\
& & \!\medcirc\! &\!\!\!\! \!\!\!\!\!\!\!\!\!\!\!\alphaup_2}
\end{equation}
\begin{equation}\tag{$\mathbf{E}\mathrm{I\!{X}}$}
  
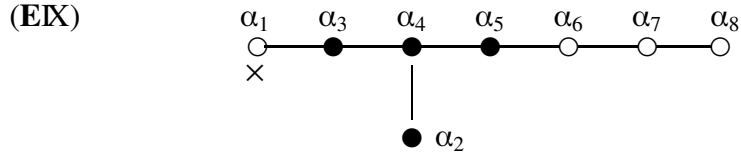
\xymatrix@C=.9pc@R=-.3pc{  \alphaup_1 & \alphaup_3 &\alphaup_4 &\alphaup_5 &\alphaup_6
  &\alphaup_7&\alphaup_8\\
\!\!\medcirc\!\!\!
\ar@{-}[r]&\!\!\medbullet\!\! \ar@{-}[r]&\!\!\medbullet\!\!\ar@{-}[r]  \ar@{-}[dddddd]
&\!\!\medbullet\!\!\ar@{-}[r]  &\!\!\medcirc\!\!\ar@{-}[r] &\!\!\medcirc\!\!\ar@{-}[r] &\!\!\medcirc\\
\times &&  \\ && \\ && \\ && \\ && \\
& & \!\medbullet\! &\!\!\!\! \!\!\!\!\!\!\!\!\!\!\!\alphaup_2}
\end{equation}
The corresponding simple ideals in the degree $0$ subalgebras 
are $\ot(7,7)$ and $\ot(3,11),$ respectively.
\subsubsection{Complex spin representations of real Lie algebras 
of type $\mathbf{D}$}
The complexification of an irreducible representation of the complex
type
of 
a real Lie algebra  is the direct sum of two non isomorphic representations. We should therefore start
by considering the complex effective prolongations of type $\gt_0$ for a reductive $\gt_0$
with $[\gt_0,\gt_0]{\simeq}\ot(2m,\C)$ and a two-dimensional center,
having a $\gt_{{-}1}{\simeq}S^{\C}_{-}(m){\oplus}S^{\C}_{+}(m).$ We need to add
a marker $\epi_{\pm}$ for each irreducible component $S^{\C}_{\pm}(m)$ of $\gt_{{-}1}.$
To find a semisimple effective prolongation, we need that $\omegaup_{m{-}1}{+}\epi_-$ and
$\omegaup_{m}{+}\epi_+$ have length $\sqrt{2}$ and are orthogonal to each other.
Then $m{\in}\{4,5,6,7\}$ and 
\begin{equation*}
 \|\epi_{\pm}\|^2=\frac{8{-}m}{4},\qquad (\epi_{+}|\epi_{-}){=}-(\omegaup_{+}|\omegaup_{-})=
 \frac{m-2}{4}.
\end{equation*}
 By Cauchy's ineguality we need that 
\begin{equation*}
 \frac{m{-}2}{4}<\frac{(8-m)^2}{16}.
\end{equation*}
This is possible only for $m{=}4.$ Thus we assume $m{=}4$ and take markers $\epi_{\pm}$ with
\begin{equation*}
 \|\epi_{\pm}\|^2=1,\;\; \; (\epi_{+}|\epi_{-})={-}\frac{1}{2}.
\end{equation*}
Then we define 
\begin{align*}
&
 \begin{cases}
 \Rad_{\;{-}2}^{(6)}
 =\{\epi_-{+}\epi_+{\pm}\eq_i\mid 1\leq{i}\leq{4}\},\\
 \Rad_{\;{-}1}^{(6)}=\{\epi_-{+}\wq\mid\wq{\in}\Wi_{-}(4)\} 
 \cup \{\epi_+{+}\wq\mid\wq{\in}\Wi_{+}(4)\},\\
 \Rad_{\;0}^{(6)}=\{{\pm}\eq_i{\pm}\eq_j\mid 1{\leq}i{<}j{\leq}4\},\\
 \Rad_{\;1}^{(6)}=\{{-}\epi_{-}{+}\wq\mid\wq{\in}\Wi_{-}(4)\} 
 \cup \{{-}\epi_+{+}\wq\mid\wq{\in}\Wi_{+}(4)\},\\ 
  \Rad_{\;2}^{(6)}=\{{-}\epi_{-}
  {-}\epi_{+}{\pm}\eq_i\mid 1\leq{i}\leq{4}\},
\end{cases}
\end{align*}
One can check that $\Rad^{(6)}={\bigcup}\Rad^{(6)}_{\;\pq}$ is a root system of type $\mathbf{E}_6,$
and 
\begin{equation*}
 \gt{=}{\sum}_{\pq{=}{-}2}^{2}\gt_{\pq},\;\;\text{with}\;\; 
\begin{cases}
 \gt_0{=}\hg_n{\oplus}\langle\Rad_{\;0}^{(6)}\rangle,\;\;\dim_{\C}\hg_6{=}6,\\
 \gt_{\pq}=\langle\Rad_{\;\pq}^{(6)}\rangle,\;\;\;\pq{\neq}0,
\end{cases}
\end{equation*}
is the maximal effective prolongation of type $\gt_0,$ with $\gt_0$ reductive and $[\gt_0,\gt_0]{\simeq}\ot(8,\C),$
of a 
FGLA with $\gt_{{-}1}{\simeq}S^{\C}_-(4){\oplus}S^{\C}_+(4).$  \par
The real forms of $\ot(8,\C)$ having a complex spin representation are $\ot(3,5)$ and $\ot(1,7),$
corresponding to the 
 two real non compact forms of $\mathbf{E}_6$ having 
 cross-marked Satake diagrams
\par\bigskip
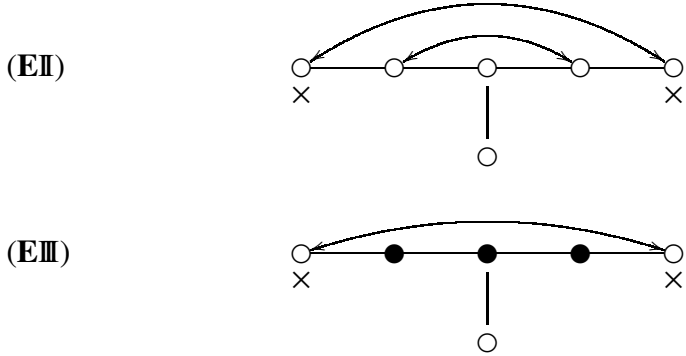
\begin{equation}\tag{$\mathbf{E}\mathrm{I\!{I}}$}
  \xymatrix@R=-.3pc{\!\!\medcirc\!\!\ar@{-}[r]\ar@{<->}@/^2pc/[rrrr]
&\!\!\medcirc\!\!
\ar@{-}[r]\ar@{<->}@/^1pc/[rr]&\!\!\medcirc\!\! \ar@{-}[r]
\ar@{-}[dddddd]&\!\!\medcirc\!\! \ar@{-}[r]&\!\!\medcirc\!\!\\
\times &&& &\times \\ && \\ && \\  && \\    &&\\ &&\medcirc
}
\end{equation}
\begin{equation}\tag{$\mathbf{E}{\mathrm{I\!{I}\!{I}}}$}
  \xymatrix@R=-.3pc{\!\!\medcirc\!\!\ar@{-}[r]\ar@{<->}@/^1pc/[rrrr]
&\!\!\medbullet\!\!
\ar@{-}[r]&\!\!\medbullet\!\! \ar@{-}[r]
\ar@{-}[dddddd]&\!\!\medbullet\!\! \ar@{-}[r]&\!\!\medcirc\!\!\\
\times &&&&\times \\ && \\ && \\  && \\    &&\\ &&\medcirc
}
\end{equation}
(the irriducible complex representation is pictured in the diagram by two crossed white nodes
joined by an arrow).
Note that the summands on degree ${\pm}2$ in $\gt$ are the vector representations
$V^{\C}_4{\simeq}\C^8$ in the complex
and $V^{\R}_4{\simeq}\R^8$ in the real cases. 
\subsubsection{Quaternionic spin representations of real Lie algebras of~type~$\mathbf{D}$}
Let $\etaup_1,\etaup_2$ be an orthonormal basis  of $\R^2.$
We denote by $\{{\pm}(\etaup_1{-}\etaup_2)\}$ the root system
of $\slt_2(\C),$ and
by $\omegaup{=}\tfrac{1}{2}(\etaup_1{-}\etaup_2)$
its fundamental weight. \par 
Then 
$\{{\pm}(\etaup_1{-}\etaup_2)\}
{\cup}\{{\pm}\eq_i{\pm}\eq_j{\mid}1{\leq}i{<}j{\leq}m\}$
is the root system
of $\slt_2(\C){\oplus}\ot(2m,\C)$
 and, taking the usual lexicographic orders,  
its fundamental weights are 
$\omegaup,\omegaup_1,\hdots,\omegaup_m.$ \par 
The complexification of an irreducible 
spin representation of quaternionic type of a real form of 
$\ot(2m,\C)$ 
lifts to an irreducible complex representation 
of $\slt_2(\C){\oplus}\ot(2m,\C)$
whose dominant weight is either $\omegaup{+}\omegaup_{m{-}1},$ or
 $\omegaup{+}\omegaup_{m}.$ 
 Let us assume it is  $\omegaup{+}\omegaup_{m}.$
 A necessary condition to find a semisimple complex Lie algebra
 which is an effective prolongation of type $\gt_0,$ 
 with $$[\gt_0,\gt_0]{=}\slt_2(\C){\oplus}\ot(2m,\C)$$
 of an $\mt$ with $\gt_{-1}$ 
 equal to the irreducible $\slt_2(\C){\oplus}\ot(2m,\C)$-module with
 dominant weight $\omegaup{+}\omegaup_m$ is that 
 \vspace{-6pt}
\begin{equation*}\vspace{-3pt}
 \|\omegaup{+}\omegaup_{m}\|^2=\frac{1}{2}{+}\frac{m}{4}<2,
\end{equation*}
i.e. that $m{=}4,5.$ 
With  $\epi_4{=}\tfrac{1}{2}(\etaup_1{+}\etaup_2),$  
$\epi_5{=}\tfrac{1}{\sqrt{8}}(\etaup_1{+}
\etaup_2),$ set 
\begin{equation*} \!
\begin{cases}
 \Rad^{(4)}_{\;{-}2}=\{2\epi_4\},\\
 \Rad^{(4)}_{\;{-}1}=\{\epi_4{\pm}\etaup
 {+}\wq\mid \wq{\in}\Wi_{+}(4)\},\\
 \Rad^{(4)}_{\;0}=\{\pm(\etaup_1{-}\etaup_2)\}
 {\cup}\{ {\pm}\eq_i {\pm}\eq_j
 {\mid} \begin{smallmatrix} 1{\leq}i{<}j{\leq}4
 \end{smallmatrix}\!\},\\
 \Rad^{(4)}_{\;1}=\{{-}\epi_4{\pm}\etaup{+}\wq
 \mid \wq{\in}\Wi_{+}(4)\},\\
 \Rad^{(4)}_{\;2}=\{-2\epi_4\}, 
\end{cases} \!\!\!\!
\begin{cases}
 \Rad^{(5)}_{\;{-}2}=\{2\epi_5{\pm}\eq_1\mid 1{\leq}i{\leq}\},\\
 \Rad^{(5)}_{\;{-}1}=\{\epi_5{\pm}\etaup
 {+}\wq\mid \wq{\in}\Wi_{+}(5)\},\\
 \Rad^{(5)}_{\;0}=\{\pm(\etaup_1{-}\etaup_2)\}
 {\cup}\{ {\pm}\eq_i {\pm}\eq_j{\mid} 
 \begin{smallmatrix}
 1{\leq}i{<}j{\leq}5
 \end{smallmatrix}\!
 \},\\
 \Rad^{(5)}_{\;1}=\{{-}\epi_5{\pm}\etaup
 {+}\wq\mid \wq{\in}\Wi_{+}(5)\},\\
 \Rad^{(5)}_{\;2}=\{-2\epi_5{\pm}\eq_1
 \mid 1{\leq}i{\leq}5\}.
\end{cases}
\end{equation*}
Then $\Rad^{(4)}{=}{\bigcup}_{\pq{=}{-}2}^{2}\Rad^{(4)}_{\;\pq}$ 
is a root system of type $\mathbf{D}_6$
and  $\Rad^{(5)}{=}{\bigcup}_{\pq{=}{-}2}^{2}\Rad^{(5)}_{\;\pq}$ 
a root system of type $\mathbf{E}_7.$ 
Set
\begin{equation*}
 \gt^{(m)}{=}{\sum}_{\pq{=}{-}2}^{2}\gt_{\pq}^{(m)},\;\;
 \text{with}\;\; 
\begin{cases}
 \gt^{(m)}_0{=}\hg_m{\oplus}\langle\Rad^{(m)}_{\;0}
 \rangle,\;\;\dim_{\C}\hg_m{=}m{+}2,\\
 \gt_{\pq}^{(m)}=\langle\Rad^{(m)}_{\;\pq}\rangle,\;\;\;\pq{\neq}0,
\end{cases}
\end{equation*}
for $m{=}4,5.$ Then 
$\gt^{m}_0{=}\slt_2(\C){\oplus}\ot(2m)
{\oplus}\langle\varpi\rangle,$ where $\varpi$ satisfies
$[\varpi,X]{=}\pq{\cdot}{X}$ for $X{\in}\gt_{\pq}$ 
and the $\gt^{(m)}$ are maximal effective prolongation of type 
$\gt_0^{(m)}.$   \par
For $m{=}4$ each of the summands $\gt^{(4)}_{{\pm}{1}}$ 
consists of two copies of 
$S^{\C}_+(4)$ and each of $\gt^{(4)}_{{\pm}{2}}$ 
is the scalar representation. \par
For $m{=}5$ each of the summands $\gt^{(4)}_{{\pm}{1}}$ 
consists of two copies of 
$S^{\C}_+(5)$ and each of $\gt^{(4)}_{{\pm}{2}}$ 
is the complex vector representation $V^{\C}_m{\simeq}
\C^{2m}.$\par 
The only real form of $\ot(8,\C)$  
having a quaternionic spin representation
is $\ot(2,6).$ Its prolongation is the real form of $\gt^{(4)},$ 
isomorphic to $\ot(4,8),$ 
whose graded structure  is represented 
by the cross-marked Satake diagram
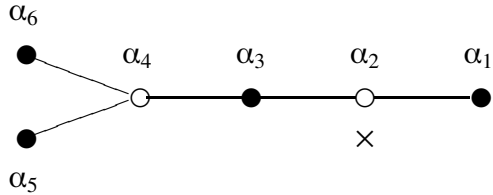
\begin{equation*}
 \xymatrix@R=.1pc{ \alphaup_{6} \\
\!\!\medbullet\!\!\! \ar@{-}[rd] &\alphaup_{4} 
&\alphaup_{3} & \alphaup_2 & \alphaup_{1}\\
& \!\!\!\medcirc\!\!\! \!\ar@{-}[r]  
& \!\!\!\medbullet\!\!\! \ar@{-}[r] &\!\!\medcirc\!\!\ar@{-}[r] &
 \!\!\medbullet\!\!\!
\\
\!\!\medbullet\!\!\! \ar@{-}[ru] & & & \times
\\ 
\alphaup_5 & & \\
 }
\end{equation*}
The only real form of $\ot(10,\C)$ 
having a quaternionic spin representation is $\ot(3,7)$ 
and the corresponding real form of
$\gt^{(5)}$ is of type $\mathbf{E}\mathrm{V\!{I}},$ 
with cross-marked Satake diagram 
 \begin{equation*} 
  \xymatrix@C=.9pc@R=-.3pc{  \alphaup_1 & \alphaup_3 
  &\alphaup_4 &\alphaup_5 &\alphaup_6
  &\alphaup_7\\
\!\!\medcirc\!\!\!
\ar@{-}[r]&\!\!\medcirc\!\! \ar@{-}[r]&\!\!
\medcirc\!\!\ar@{-}[r]  \ar@{-}[dddddd]
&\!\!\medbullet\!\!\ar@{-}[r]  &\!\!
\medcirc\!\!\ar@{-}[r] &\!\!\medbullet\\
&&  && \times 
\\ && \\ && \\ && \\ && \\
& & \!\medbullet\! &\!\!\!\! \!\!\!\!\!\!\!\!\!\!\!\alphaup_2}
\end{equation*}

\bibliographystyle{amsplain}
\renewcommand{\MR}[1]{}
\bibliography{homog}
\end{document}